\documentclass[11pt,a4paper,reqno]{amsart} 

\usepackage [english]{babel} 

\DeclareFontFamily{U}{mathb}{\hyphenchar\font45}
\DeclareFontShape{U}{mathb}{m}{n}{
 <5> <6> <7> <8> <9> <10>
 <10.95> <12> <14.4> <17.28> <20.74> <24.88>
 mathb10
 }{}
\DeclareSymbolFont{mathb}{U}{mathb}{m}{n}

\DeclareMathSymbol{\sqbullet}{1}{mathb}{"0D}

\usepackage{amsmath}
\usepackage{amssymb}
\usepackage{euscript}
\usepackage{graphicx}
\usepackage[all]{xy}
\usepackage[dvipsnames]{xcolor}
\usepackage[colorlinks=true,linktocpage=true,pagebackref=false, citecolor=black,linkcolor=black]{hyperref}
\usepackage{pifont}
\usepackage{tikz,pgfplots}
\pgfplotsset{compat=1.10}
\usepgfplotslibrary{fillbetween}
\usetikzlibrary{cd,arrows,decorations.pathmorphing,backgrounds,automata,positioning,fit,matrix}
\usepackage{caption,subcaption}
\usepackage{comment}
\usepackage{xcolor}

\allowdisplaybreaks

\pgfplotsset{soldot/.style={color=black,only marks,mark=*}} \pgfplotsset{holdot/.style={color=black,fill=white,only marks,mark=*}}

\newtheorem{thm}{Theorem}[section]

\newtheorem{lem}[thm]{Lemma}
\newtheorem{prop}[thm]{Proposition}

\newtheorem{cor}[thm]{Corollary}
\newtheorem{conj}[thm]{Conjecture}

\theoremstyle{definition}
\newtheorem{defn}[thm]{Definition}

\theoremstyle{remark}
\newtheorem{remark}[thm]{Remark}
\newtheorem{assumption}[thm]{Assumption}
\newtheorem{remarks}[thm]{Remarks}
\newtheorem{example}[thm]{Example}

\newtheorem{quest2}[thm]{Question}

\numberwithin{equation}{section}
\numberwithin{figure}{section}

 \newcommand{\N}{{\mathbb N}}
 \newcommand{\R}{{\mathbb R}}
 \newcommand{\C}{{\mathbb C}}

\newcommand{\sph}{{\mathbb S}}

 \newcommand{\Cont}{{\mathcal C}}

\newcommand{\Ff}{{\EuScript F}}

\newcommand{\Ss}{{\EuScript S}}
\newcommand{\Tt}{{\EuScript T}}
\newcommand{\Ee}{{\EuScript E}}

\newcommand{\Bb}{{\EuScript B}}
\newcommand{\Cc}{{\EuScript C}}

\newcommand{\Qq}{{\EuScript Q}}

\newcommand{\Hh}{{\EuScript H}}

\newcommand{\Nn}{{\EuScript N}}

\newcommand{\Rr}{{\EuScript R}}

\newcommand{\Reg}{\operatorname{Reg}}
\newcommand{\Sing}{\operatorname{Sing}}
\newcommand{\Int}{\operatorname{Int}}
\newcommand{\cl}{\operatorname{Cl}}
\newcommand{\dist}{\operatorname{dist}}
\newcommand{\id}{\operatorname{id}}

\newcommand{\im}{\operatorname{Im}}
\newcommand{\Sth}{\operatorname{Sth}}

\newcommand{\atsr}{\Cont^\mu\text{-}\mathtt{ats}}
\newcommand{\atsinfty}{\Cont^\infty\text{-}\mathtt{ats}}

\newcommand{\x}{{\tt x}} \newcommand{\y}{{\tt y}} 
 \renewcommand{\t}{{\tt t}}

\newcommand{\veps}{\varepsilon}
\newcommand{\eps}{\epsilon}
\newcommand{\ol}{\overline}

\usepackage{scalerel}

\begin{document}
\title[Nash approximation of differentiable semialgebraic maps]{Nash approximation of differentiable\\ semialgebraic maps}

\begin{abstract}
Let $\Tt\subset\R^n$ be a semialgebraic set and let $\mu\ge0$ be a non-negative integer. We say that $\Tt$ is a {\em Nash $\mu$-approximation target space} (or a $({\mathcal N},\mu)$-${\tt ats}$ for short) if it has the following universal approximation property: {\em For each $m\in\N$ and each locally compact semialgebraic subset $\Ss\subset\R^m$, the set of Nash maps ${\mathcal N}(\Ss,\Tt)$ is dense in the topological space ${\mathcal S}^\mu(\Ss,\Tt)$ of $\Cont^\mu$ semialgebraic maps between $\Ss$ and $\Tt$}. A necessary condition to be a $({\mathcal N},\mu)$-${\tt ats}$ is that $\Tt$ is locally connected by analytic paths. In this paper we show: {\em Nash manifolds with corners are $({\mathcal N},\mu)$-${\tt ats}$ for each $\mu\geq0$}. As an application of a stronger version of the previous statement, we show that if two Nash maps $f,g:\Ss\to\Qq$, where $\Ss$ is a locally compact semialgebraic set of $\R^m$ and $\Qq$ is a Nash manifold with corners, are close enough in the (strong) Whitney's semialgebraic topology of ${\mathcal S}^0(\Ss,\Tt)$ (and consequently they are continuous semialgebraically homotopic), then $f,g$ are Nash homotopic.
\end{abstract}

\date{03/08/2026}
\author{Antonio Carbone}
\address{Dipartimento di Scienze dell'Ambiente e della Prevenzione, Palazzo Turchi di Bagno, C.so Ercole I D'Este, 32, Università di Ferrara, 44121 Ferrara (ITALY)}
\email{antonio.carbone@unife.it}
\thanks{The first author is supported by GNSAGA of INDAM}

\author{Jos\'e F. Fernando}
\address{Departamento de \'Algebra, Geometr\'\i a y Topolog\'\i a, Facultad de Ciencias Matem\'aticas, Universidad Complutense de Madrid, Plaza de Ciencias 3, 28040 MADRID (SPAIN)}
\email{josefer@mat.ucm.es}
\thanks{The second author is supported by PID2021-122752NB-I00}
\maketitle

\section{Introduction}\label{s1}

Approximation is a tool of great importance in many areas of mathematics. It allows one to understand objects and morphisms of a certain category taking advantage of the corresponding properties of objects and morphisms in other categories that enjoy a better behavior and are dense inside the one we want to study. The importance of approximation of continuous functions is a natural question that arises from the Stone-Weierstrass result on uniform approximation of continuous functions over compact sets by polynomial functions. This result provides naturally a polynomial approximation result for $\R^n$-valued continuous maps $f:K\to\R^n$ defined on a compact subset $K\subset\R^m$. Difficulties arise when one tries to restrict the image of the approximating map. One way to proceed is to be more flexible with the type of approximating maps but also by considering domains of definition and target spaces in suitable tame categories.

This paper deals with a Nash approximation problem in the differentiable and continuous semialgebraic category and we also present some applications of the approximation results to semialgebraic homotopies. Recall that a set $\Ss\subset\R^m$ is \em semialgebraic \em if it is a Boolean combi\-nation of sets defined by polynomial equalities and inequalities. We endow $\R^m$ with the Euclidean topology and define the {\em distance between two points} $x,y\in\R^m$ as $d(x,y):=\|x-y\|$. The category of semialgebraic sets is closed under basic boolean operations but also under usual topological operations: taking closures (denoted by $\cl(\cdot)$), interiors (denoted by $\Int(\cdot)$), connected components, etc. Given a set $A\subset\R^m$ and a point $x\in\R^m$, we define $\dist(x,A):=\inf\{d(x,y):\ y\in A\}$. It holds $\dist(x,A)=0$ if and only if $x\in\cl(A)$. In addition, if $K\subset\R^m$ is a compact set, we define $\dist(K,A):=\min\{\dist(x,A):\ x\in K\}=\inf\{d(x,y):\ x\in K, y\in A\}$.

Let $\Ss\subset\R^m$ and $\Tt\subset\R^n$ be (non-empty) semialgebraic sets. A map $f:\Ss\to\Tt$ is \em semialgebraic \em if its graph is a semialgebraic subset of $\R^{m+n}$. In particular, given a semialgebraic set $A\subset\R^m$ the function $\dist(\cdot,A):\R^m\to[0,+\infty)$ is a continuous semialgebraic function whose zero set is $\cl(A)$. A {\em Nash map} $f:\Omega\to\R^n$ on an open semialgebraic set $\Omega\subset\R^m$ is both a $\Cont^\infty$ and a semialgebraic map on $\Omega$. By \cite[Prop.8.1.8]{bcr} these maps coincide with those that are both analytic and semialgebraic maps on $\Omega$. Let us denote the ring of Nash functions $f:\Omega\to\R$ on $\Omega$ by $\Nn(\Omega)$.

More generally, let $\Ss\subset\R^m$ be a semialgebraic set. We say that a function $f:\Ss\to\R$ is a \em Nash function \em if there exist an open semialgebraic neighborhood $\Omega\subset\R^m$ of $\Ss$ and a Nash extension $F:\Omega\to\R$ of $f$. We denote by $\Nn(\Ss)$ the ring of Nash functions on $\Ss$ and we refer the reader to \cite{bfr,fg2,fgh3} for further results about this ring. A map $f:=(f_1,\ldots,f_n):\Ss\to\Tt$ between semialgebraic sets $\Ss\subset\R^m$ and $\Tt\subset\R^n$ is {\em Nash} if its components $f_i:\Ss\to\R$ are Nash functions on $\Ss$. Denote~by $\Nn(\Ss,\Tt)$ the set of Nash maps on $\Ss$ that have $\Tt$ as target space.

Let $\mu\geq0$ be a non-negative integer. A function on the open semialgebraic set $\Omega\subset\R^m$ is ${\mathcal S}^\mu$ if it is $\Cont^\mu$ and semialgebraic. Let $\Ss\subset\R^m$ be a semialgebraic set. We say that $f:\Ss\to\R$ is a \em ${\mathcal S}^\mu$ function \em if there exist an open semialgebraic neighborhood $\Omega\subset\R^m$ of $\Ss$ and an ${\mathcal S}^\mu$ extension $F:\Omega\to\R$ of $f$. We denote by ${\mathcal S}^\mu(\Ss)$ the ring of ${\mathcal S}^\mu$ functions on $\Ss$ and we refer the reader to \cite{bfg} for further results about this ring. The case $\mu=0$ corresponds to the ring ${\mathcal S}^0(\Ss)$, which in general differs from the ring of continuous semialgebraic functions on $\Ss$, see \cite[Thm.1.1]{bfg} and \cite[Thm.5]{fe1}. However, if $\Ss$ is locally compact both rings coincide as a consequence of \cite{dk}. A map $f:=(f_1,\ldots,f_n):\Ss\to\Tt$ between semialgebraic sets $\Ss\subset\R^m$ and $\Tt\subset\R^n$ is ${\mathcal S}^\mu$ if its components $f_i:\Ss\to\R$ are ${\mathcal S}^\mu$ functions on $\Ss$. Denote~by ${\mathcal S}^\mu(\Ss,\Tt)$ the set of ${\mathcal S}^\mu$ maps on $\Ss$ that have $\Tt$ as target space. 

In \cite{kp2,at,th} the authors made a careful analysis of an intrinsic definition of ${\mathcal S}^\mu$-functions in terms of jets of order $\mu$ of (continuous) semialgebraic functions \cite[Def.1.1]{at}. If in addition $m=2$, Fefferman-Luli proved in \cite{fl} that $\Cont^\mu$ functions on a closed semialgebraic set that are semialgebraic (intrinsic description) coincide with ${\mathcal S}^\mu$ functions (extrinsic description). In \cite[Cor.1.7]{bcm} Bierstone-Campesato-Milman presented a weak version for any dimension $m$ of the implication between the intrinsic and extrinsic descriptions of ${\mathcal S}^\mu$ functions, via an intermediate function $t:\N\to\N$, which encodes {\em a certain loss of differentiability}. Namely, if $f:\Ss\to\R$ is a semialgebraic function on a closed semialgebraic set $\Ss\subset\R^n$ and there exist an open semialgebraic neighborhood $\Omega\subset\R^n$ of $\Ss$ and a $\Cont^{t(p)}$ differentiable function $G:\Omega\to\R$ such that $G|_\Ss=f$, then there exists an ${\mathcal S}^p$-function $F:\Omega\to\R$ such that $F|_\Ss=f$. The previous result is used in \cite[Appendix B]{fgh3} to analyze the equality between smooth semialgebraic functions and Nash functions on semialgebraic sets.

A semialgebraic set $\Tt\subset\R^n$ is a {\em Nash $\mu$-approximation target space} (or a $({\mathcal N},\mu)$-${\tt ats}$ for short) if it has the following universal approximation property: {\em For each $m\in\N$ and each locally compact semialgebraic subset $\Ss\subset\R^m$, the set of Nash maps ${\mathcal N}(\Ss,\Tt)$ is dense in the topological space ${\mathcal S}^\mu(\Ss,\Tt)$ of ${\mathcal S}^\mu$ maps between $\Ss$ and $\Tt$} (endowed with the ${\mathcal S}^\mu$ topology that we recall in~\S\ref{wst}).

A semialgebraic set $M\subset\R^n$ is called an (affine) \em Nash manifold \em if it is in addition a smooth (that is, $\Cont^\infty$) submanifold of (an open semialgebraic subset of) $\R^n$. By \cite[Prop.8.1.8 \& Cor.9.3.10]{bcr} Nash manifolds coincide with analytic submanifolds of (an open semialgebraic subset of) $\R^n$ that are themselves semialgebraic sets. A semialgebraic set $\Hh\subset\R^n$ is an (affine) {\em Nash manifold with boundary} if it is in addition a smooth (that is, $\Cont^\infty$) manifold with boundary. In particular, the boundary $\partial\Hh$ of $\Hh$ is a semialgebraic set that is in addition a smooth manifold, so it is a Nash manifold. A semialgebraic set $\Qq\subset\R^n$ is an (affine) {\em Nash manifold with corners} if it is in addition a smooth (that is, $\Cont^\infty$) manifold with corners. We will come back to Nash manifolds with corners in Section \ref{s3}.

\subsection{State of the art}
In the literature there are many approximation results of algebraic/semi\-algebraic nature determined by the geometry of the target space, as it happens with the main results of this article. We recall here some of them, which are very representative about the actual picture of this type of problems.

\subsubsection*{${\mathcal S}^\mu$ maps by Nash maps} 
Efroymson's approximation theorem \cite[\S1]{ef} ensures that continuous semialgebraic functions can be approximated by Nash functions on a Nash manifold. This statement was improved by Shiota in many directions \cite{sh}, for instance, providing a similar approximation result for ${\mathcal S}^\mu$ functions using (strong) Whitney's ${\mathcal S}^\mu$ topology and proving relative versions of such an ${\mathcal S}^\mu$ approximation result. The previous results can be extended to approximate ${\mathcal S}^\mu$ maps $f:\Ss\to N$ from a locally compact semialgebraic set $\Ss\subset\R^m$ to a Nash manifold $N\subset\R^n$ by Nash maps $g:\Ss\to N$ in the ${\mathcal S}^\mu$ topology, making use of a suitable Nash tubular neighborhood of $N$ in $\R^n$ (see \cite[Cor.8.9.5]{bcr} and \cite[Lem.I.3.2]{sh}), in other words, Nash manifolds are $({\mathcal N},\mu)$-${\tt ats}$ for each $\mu\geq0$. In addition, there are approximation results for ${\mathcal S}^\mu$ maps between Nash sets with monomial singularities (see \cite[Thm.1.7]{bfr} and also \cite{bfr,fe2,fgr} for further applications of this type of approximation results), however the latter results do not allow to decide whether Nash sets with monomial singularities are or not $({\mathcal N},\mu)$-${\tt ats}$ for some $\mu\geq0$.

\subsubsection*{${\mathcal S}^0$ maps by ${\mathcal S}^\mu$ maps}
In \cite{fgh1} the problem of uniform approximation of continuous semialgebraic maps $f:\Ss\to\Tt$ from a compact semialgebraic set $\Ss$ to an arbitrary semialgebraic set $\Tt$ by ${\mathcal S}^\mu$ maps $g:\Ss\to\Tt$ for a fixed integer $\mu\geq1$ is approached. The techniques developed in \cite{fgh1} together with Paw\l ucki's $\Cont^\mu$-triangulations results presented in \cite{os,p1,p2} allow to prove that such a uniform approximation is always possible. In \cite{ca} the first author made use of Paw\l ucki's techniques of $\Cont^\mu$-triangulation \cite{p1} to show that all continuous semialgebraic maps between compact semialgebraic sets can be approximated by differentiable semialgebraic maps without changing their image after the approximation. The argument is an interplay between semialgebraic geometry and PL geometry and makes use of a `surjective semialgebraic version' of the finite simplicial approximation theorem.

\subsubsection*{$\Cont^0$ maps by $\Cont^\mu$ maps}
Let $Y\subset\R^n$ be a triangulable set and let $\mu$ be either a positive integer or $\mu=\infty$. We say that $Y$ is a $\Cont^\mu$-approximation target space, or a $\atsr$ for short, if it has the following universal approximation property: \em For each $m\in\N$ and each locally compact subset $X$ of~$\R^m$, each continuous map $f:X\to Y$ can be approximated by $\Cont^\mu$ maps $g:X\to Y$ with respect to the (strong) Whitney's $\Cont^0$ topology\em. Using approximation techniques developed in \cite{fgh2} it holds: \em if $Y$ is weakly $\Cont^\mu$ triangulable, then $Y$ is a $\atsr$\em. This result applies to relevant classes of triangulable sets, namely: (1) every locally compact polyhedron is a $\atsinfty$, (2) every set that is locally $\Cont^\mu$ equivalent to a polyhedron is a $\atsr$ (this includes $\Cont^\mu$ submanifolds with corners of $\R^n$) and (3) every locally compact locally definable set of an arbitrary o-minimal structure is by \cite{p1,p2} a $\atsr$ (this includes locally compact locally semialgebraic sets and locally compact subanalytic sets) for each $\mu\geq1$. In addition: \em if $Y$ is a global analytic set, then each proper continuous map $f:X\to Y$ can be approximated by proper $\Cont^\infty$ maps $g:X\to Y$\em. 

\subsubsection*{$\Cont^{\infty}$ maps by regular maps}
A non-singular real algebraic set $Y$ is said to have the {\em approximation property} if for every non-singular real algebraic set $X$ the following property holds: {\em if $f:X\to Y$ is a $\Cont^{\infty}$ map that is homotopic to a regular map, then $f$ can be approximated in the $\Cont^{\infty}$ topology by regular maps}. There is a large literature concerning this topic, but we focus on \cite{bak}, where Banecki and Kucharz fully characterize the varieties $Y$ with the approximation property. The authors also characterize the varieties $Y$ with the approximation property combined with a suitable interpolation condition. Some of the previous results have variants concerning the regular approximation of continuous maps defined on (possibly singular) real algebraic sets. In the same vein we refer the reader to \cite{bk1}.

Let $X$ and $Y$ be non-singular real algebraic sets of positive dimension such that $X$ is compact. The set ${\mathcal R}(X,Y)$ of regular maps from $X$ to $Y$ turns out to be dense in the corresponding space ${\mathcal N}(X,Y)\subset\Cont^\infty(X,Y)$ of Nash maps endowed with the ${\mathcal C}^\infty$ compact-open topology only in exceptional cases. If $Y$ is `generic' in a suita\-ble way, ${\mathcal R}(X,Y)$ is an `extremely small' closed subset of ${\mathcal N}(X,Y)$, see \cite{gh06-1,gh06-2}. This lack of regular maps between non-singular real algebraic sets seems to be the main obstru\-ction for an extension of the Nash-Tognoli algebraization techniques from smooth manifolds to singular polyhedral spaces and, in particular, to compact Nash sets, see \cite{ak,gt}. It is worthwhile mentioning that compact real algebraic sets with at least two points do not have tubular neighborhoods with regular retractions \cite[Thm.2]{gh0}.

\subsubsection*{$\Cont^{\infty}$ maps by $\mu$-regulous maps}

Let $X$, $Y$ be non-singular real algebraic sets. A map $f:X\to Y$ is {\em $\mu$-regulous}, where $\mu\geq0$ is a non-negative integer, if it is of class $\Cont^\mu$ and the restriction of $f$ to some Zariski open dense subset of $X$ is a regular map. A real algebraic set $Y$ of dimension $d$ is {\em uniformly rational} if each of its points has a Zariski open neighborhood which is biregularly isomorphic to a Zariski open subset of $\R^d$. Assuming that $Y$ is uniformly rational, and $\mu>0$, Kucharz proves in \cite{kz3} that a $\Cont^\infty$ map $f:X\to Y$ can be approximated by $\mu$-regulous maps in the $\Cont^\mu$ topology if and only if $f$ is homotopic to a $\mu$-regulous map. The class of uniformly rational real algebraic varieties includes spheres, Grassmannians and real rational surfaces, and is stable under blowing up non-singular centers. Furthermore, taking $Y=\sph^p$ (the unit $p$-dimensional sphere), the author obtains several new results on approximation of $\Cont^\infty$ maps from $X$ into $\sph^p$ by $\mu$-regulous maps in the $\Cont^\mu$ topology, for $\mu\geq0$.

More generally, given a finite simplicial complex $K$ in $\R^n$ and a real algebraic set $Y$, by a {\em $K$-regular map} $|K|\to Y$ we mean a continuous map whose restriction to every simplex in $K$ is a regular map. A real algebraic set $Y$ is {\em uniformly retract rational} if for every point $y\in Y$ there is a Zariski open neighborhood $V\subset Y$ of $y$ such that the identity map of $V$ is the composite of regular maps $V\to W\to V$, where $W\subset\R^d$ is a Zariski open set for some $d$ depending on $y$. A simplified version of the main result in \cite{bik} says that if $Y$ is a uniformly retract rational set and if $0\leq\nu\leq\mu$ are integers, every $\Cont^\nu$ map $|K|\to Y$ can be approximated in the $\Cont^\nu$ topology by $K$-regular maps of class $\Cont^\mu$. See also the following articles by Banecki \cite{ba1,ba2,ba3,ba4}.
 
\subsubsection*{$\Cont^0$ maps by continuous rational maps} 
Kucharz has studied deeply approximation results of continuous maps between a compact non-singular real algebraic set $X$ and a sphere $\sph^n$ by continuous rational maps. As $X$ is compact the author considers on the space $\Cont^0(X,\sph^n)$ of continuous maps from $X$ to $\sph^n$ the compact-open topology. In case $X$ has dimension $n$, the space ${\mathcal R}_0(X,\sph^n)$ of (nice) continuous rational maps from $X$ to $\sph^n$ is dense in $\Cont^0(X,\sph^n)$ (see \cite[Thm.1.2, Cor.1.3]{kz2}). In addition, for any pair $(m,n)$ of non-negative integers, the set ${\mathcal R}_0(\sph^m,\sph^n)$ is dense in $\Cont^0(\sph^m,\sph^n)$ (see \cite[Thm.1.5]{kz2}). In general ${\mathcal R}_0(X,\sph^n)$ needs not to be dense in $\Cont^0(X,\sph^n)$. Simple obstructions can be expressed in terms of homology or cohomology classes representable by algebraic subsets. We refer the reader to \cite{kz1,kz2} for further details.

\subsection{Main results}
We state the main results of this work, proved in Sections \ref{s4} and \ref{s6}. Their proofs rely heavily on some tools already developed in \cite{cf2}.

\begin{thm}[Nash manifold with corners]\label{thm1}
Let $\Ss\subset\R^m$ be a locally compact semialgebraic set, let $\Qq$ be a Nash manifold with corners and let $f:\Ss\to\Qq$ be an ${\mathcal S}^{\mu}$ map. Then there exist Nash maps $g:\Ss\to\Int(\Qq)$ arbitrarily close to $f$ with respect to the ${\mathcal S}^{\mu}$ topology.

Consequently, every Nash manifold with corners is an $({\mathcal N},\mu)$-${\tt ats}$ for each $\mu\geq0$. 
\end{thm}

The following result is a generalization of \cite[II.5.2]{sh} when the target space is a Nash manifold with (divisorial) corners (see Definition \ref{nmwc} and Assumption \ref{nmwc0}). 

\begin{thm}[Relative Nash approximation]\label{thm2}
Let $N\subset\R^m$ be a Nash manifold, let $X$ be a Nash subset of $N$, let $\Ss$ be a closed semialgebraic subset of $N$ and let $\Qq\subset\R^n$ be a Nash manifold with corners. Let $\mu\geq 0$ be an integer and $U$ an open semialgebraic neighborhood of $X$ in $N$. Let $f:\Ss\to\Qq$ be an ${\mathcal S}^{\mu}$ map that is Nash on $\Ss\cap U$. Then there exist Nash maps $g:\Ss\to\Qq$ arbitrarily close to $f$ with respect to the ${\mathcal S}^{\mu}$ topology and such that $g(x)=f(x)$ for each $x\in X\cap\Ss$.
\end{thm}

The proof of Theorem \ref{thm2} strongly relies on the following result that shows how to `push' a Nash manifold with corners inside its interior using a one-parameter family of Nash maps.

\begin{thm}[Pushing a Nash manifold with corners inside its interior]\label{push}
Let $\Qq\subset\R^n$ be a Nash manifold with corners of dimension $d$. Then there exists a Nash map $\sigma:\Qq\times [0,1]\to\Qq,\, (x,t)\mapsto\sigma_t(x):=\sigma(x,t)$ such that 
\begin{itemize}
\item $\sigma_0(x)=x$ for each $x\in\Qq$, 
\item $\sigma_t(\Qq)\subset\Int(\Qq)$ for each $t\in(0,1]$.
\end{itemize}

In addition, there exist a $d$-dimensional Nash manifold $M$ that contains $\Qq$ as a closed subset, an open semialgebraic neighborhood $M_0\subset M$ of $\Qq$ and a Nash extension $\Sigma:M_0\times[-1,1]\to M$ of $\sigma:\Qq\times [0,1]\to\Qq$.
\end{thm}

The previous result is improved in Theorem \ref{pushdiffeo} involving Nash embeddings for each $t\in[0,1]$; however, its statement is rather more technical and we present it in Section \ref{s5}.

\subsection{Applications of the main results and further developments}
We apply Theorem \ref{thm2} to study Nash approximation of (semialgebraic) homotopies, which was the starting motivating point for this paper. We were initially inspired by some results in \cite[Ch.III.\S8]{orr}.

\subsubsection{Nash homotopies}
Let $\Ss\subset\R^m$ be a locally compact semialgebraic set and $\Qq\subset\R^n$ a Nash manifold with corners. We have the following results about Nash homotopies.

\begin{thm}[Nash approximation of homotopies]\label{homotopy2}
Let $f,g:\Ss\to\Qq$ be Nash maps such that there exists a continuous semialgebraic homotopy $\Phi:\Ss\times [0,1]\to\Qq$ between $f$ and $g$. Then there exist Nash homotopies $\Psi:\Ss\times[0,1]\to\Qq$ between $f$ and $g$ arbitrarily close to $\Phi$ with respect to the continuous semialgebraic topology. In addition, if there exists $\mu\geq1$ such that $\Phi_t$ is ${\mathcal S}^\mu$ for each $t\in[0,1]$, then $\Psi_t$ is close to $\Phi_t$ in the ${\mathcal S}^\mu$-topology of ${\mathcal S}^\mu(\Ss,\Qq)$ for each $t\in[0,1]$. 
\end{thm}

\begin{cor}\label{homocor1}
Let $f,g:\Ss\to\Qq$ be Nash maps. If $f$ and $g$ are close enough with respect to the continuous semialgebraic topology, then they are Nash homotopic. 
\end{cor}

\begin{cor}\label{homocor2}
Let $f:\Ss\to\Qq$ be an ${\mathcal S}^{\mu}$ map. Then, $f$ is homotopic to a Nash map $g:\Ss\to\Int(\Qq)$ and the homotopy can be chosen ${\mathcal S}^{\mu}$. 
\end{cor}

\subsubsection{Further developments}
The main problem one would like to solve consists of: {\em characterizing the $({\mathcal N},\mu)$-${\tt ats}$ semialgebraic sets for each $\mu\geq0$}. Recall that a semialgebraic set $\Ss\subset\R^n$ is {\em connected by analytic paths} if for each pair of points $x,y\in\Ss$ there exists an analytic path $\alpha:[0,1]\to\Ss$ such that $\alpha(0)=x$ and $\alpha(1)=y$. We say that the semialgebraic set $\Ss\subset\R^n$ is {\em locally connected by analytic paths} if for each $x\in\Ss$ and each open (semialgebraic) neighborhood $U^x\subset\R^n$ of $x$, there exists an open semialgebraic neighborhood $W^x\subset U^x$ of $x$ such that $\Ss\cap W^x$ is connected by analytic paths. If $p\in\R^n$ and $\veps>0$, we denote by $\ol{\Bb}_n(p,\veps)$ the {\em closed ball of center $p$ and radius $\veps$}. Its boundary $\partial\ol{\Bb}_n(p,\veps)$ is the {\em sphere $\sph^n(p,\veps)$ center $p$ and radius $\veps$}, whereas its interior is the {\em open ball ${\Bb}_n(p,\veps)$ of center $p$ and radius $\veps$}. These three families provide some examples of semialgebraic sets that are locally connected by analytic paths.

A necessary condition to be an $({\mathcal N},\mu)$-${\tt ats}$ is described in the following lemma that we prove in Section \ref{s7}.

\begin{lem}\label{fdlem1}
Let $\Tt\subset\R^n$ be a $({\mathcal N},\mu)$-${\tt ats}$ semialgebraic set. Then $\Tt$ is locally connected by analytic paths.
\end{lem}

As a consequence of \cite[Main Thm.1.8]{fe3} we also prove in Section \ref{s7} the following result. Its converse is false as we show in Example \ref{counter}.
 
\begin{lem}\label{fdlem2}
A semialgebraic set $\Tt\subset\R^n$ connected and locally connected by analytic paths is connected by analytic paths. 
\end{lem}

We expect that Theorem \ref{thm1} for Nash manifolds with corners can be generalized to semialgebraic sets locally connected by analytic paths, at least when $\mu=0$. We hope that the techniques already introduced in \cite{cf2,cf3} (which have further applications concerning \cite{cf1}) together with the results of this article (especially Theorem \ref{thm1}) will help to prove the following result when $\mu=0$:

\begin{conj}\label{carbone}
A semialgebraic set $\Tt\subset\R^n$ is an $({\mathcal N},0)$-${\tt ats}$ if and only if $\Tt$ is locally connected by analytic paths.
\end{conj}
\begin{remark}
As a representative family of semialgebraic sets related to Conjecture \ref{carbone}, Nash manifolds with corners are locally connected by analytic paths and $({\mathcal N},\mu)$-${\tt ats}$ for each $\mu\geq0$.\hfill$\sqbullet$
\end{remark}

A related question is the following:
\begin{quest2}\label{quest}
Let $0\leq\nu\leq\mu$ be integers. Does it hold that an $({\mathcal N},\mu)$-${\tt ats}$ is an $({\mathcal N},\nu)$-${\tt ats}$?
\end{quest2}

Let $\Tt\subset\R^n$ be a semialgebraic set and let $0\leq\nu\leq\mu$ be integers. Suppose $\Tt$ is an $({\mathcal N},\nu)$-${\tt ats}$. Then every ${\mathcal S}^\nu$ map $f:\Ss\to\Tt$ where $\Ss$ is a locally compact semialgebraic set can be arbitrarily approximated by Nash maps $g:\Ss\to\Tt$ in the ${\mathcal S}^\nu$ topology. Thus, each ${\mathcal S}^\nu$ map $f:\Ss\to\Tt$ where $\Ss$ is a locally compact semialgebraic set can be arbitrarily approximated by ${\mathcal S}^\mu$ maps $h:\Ss\to\Tt$ in the ${\mathcal S}^\nu$ topology. Conversely, if every ${\mathcal S}^\nu$ map $f:\Ss\to\Tt$ where $\Ss$ is a locally compact semialgebraic set can be arbitrarily approximated by ${\mathcal S}^\mu$ maps $h:\Ss\to\Tt$ in the ${\mathcal S}^\nu$ topology and $\Tt$ is an $({\mathcal N},\mu)$-${\tt ats}$, then $\Tt$ is also an $({\mathcal N},\nu)$-${\tt ats}$. Consequently, Question \ref{quest} is equivalent to {\em determining whether ${\mathcal S}^\mu(\Ss,\Tt)$ is dense in ${\mathcal S}^\nu(\Ss,\Tt)$ (with respect to the ${\mathcal S}^\nu$ topology) for each locally compact semialgebraic set $\Ss\subset\R^m$ and each $m\geq1$.} In the case $0=\nu\leq\mu$ Question \ref{quest} is still an open problem, although we know from \cite{fgh1,p1} that every ${\mathcal S}^0$ map $f:\Ss\to\Tt$ where $\Ss$ is a compact semialgebraic set can be arbitrarily approximated by ${\mathcal S}^\nu$ maps $g:\Ss\to\Tt$ in the ${\mathcal S}^0$ topology.

\begin{remark}
We are convinced that it should be possible to develop similar results in the global analytic setting. To that end one should change: (1) locally compact semialgebraic sets by locally compact sets or maybe by the more restrictive family of globally defined semianalytic sets \cite{abf}, (2) Nash sets by global analytic sets \cite{c2}, (3) Nash manifolds (resp. with boundary or more generally with corners) by analytic manifolds (resp. with boundary or more generally with corners), (4) Nash functions and maps by analytic functions and maps, (5) ${\mathcal S}^\mu$ functions and maps by ${\mathcal C}^\mu$ functions and maps, (6) Nash and ${\mathcal S}^\mu$ homotopies by analytic and ${\mathcal C}^\mu$ homotopies, among other things. Crucial tools should surely be Whitney's approximation Theorem \cite[\S1.6]{n} and Whitney's immersion Theorem \cite[\S2.15]{n}, but also the natural generalization of the tools developed in this article and in \cite{cf2,cf3}.\hfill$\sqbullet$
\end{remark}

\subsection*{Structure of the article}
The article is organized as follows. In Sections \ref{s2} and \ref{s3} we present some preliminary concepts and results that will be used in the remaining sections, with special care to the ${\mathcal S}^\mu$ topology and Nash manifolds with corners. In Section \ref{s4} we prove Theorem \ref{push}, whereas in Section \ref{s5} we improve the latter result, proving that for each integer $\mu\geq 0$ it is possible to push a Nash manifold with corners $\Qq$ inside its interior using a one-parameter family of Nash diffeomorphisms (onto their images) arbitrarily close to the identity map with respect to the ${\mathcal S}^\mu$ topology. In Section \ref{s6} we prove Theorems \ref{thm1} and \ref{thm2}, which are the main results of this article. Finally, in Section \ref{s7} we present some applications of Theorem \ref{thm2} to construct Nash homotopies between Nash maps, which was our initial purpose to achieve when writing this article. In addition, we prove Lemmas \ref{fdlem1} and \ref{fdlem2} and present an enlightening example.

\section{Whitney's semialgebraic topology}\label{s2}

The purpose of this section is to recall Whitney's semialgebraic topology and to introduce trimmed Whitney's semialgebraic topology, which is a modified version of the former to handle homotopies $H:\Ss\times[0,1]\to\Tt$ constituted by ${\mathcal S}^\mu$ maps $H_t$ between the semialgebraic sets $\Ss\subset\R^m$ and $\Tt\subset\R^n$ for each value of the parameter $t\in[0,1]$.

\subsection{Mostowski's trick}

We begin with a well-known trick of Mostowski, because we use it several times along this article to embed via a Nash map locally compact semialgebraic sets of an affine space as closed semialgebraic sets of a (possibly larger) affine space. 
Given a semialgebraic set $\Ss\subset\R^m$, we define its {\em exterior boundary} as $\delta^\bullet\Ss:=\cl(\Ss)\setminus\Ss$.

\begin{prop}[Mostowski's trick]\label{Mos}
Let $\Ss\subset\R^m$ be a locally compact semialgebraic set. Then there exists a Nash map $H:\R^m\setminus\delta^\bullet\Ss\to\R^{m+1}$ such that the image $H(\Ss)$ is a closed semialgebraic subset of $\R^{m+1}$ Nash diffeomorphic to $\Ss$. 
\end{prop}
\begin{proof}
As $\Ss$ is locally compact in $\R^m$, its exterior boundary $\delta^\bullet\Ss$ is a closed semialgebraic subset of $\R^m$. By \cite[Lem.6]{mo} there exists a continuous semialgebraic function $h:\R^m\to\R$ such that $\delta^\bullet\Ss=\{h=0\}$ and $h$ is Nash on the open semialgebraic set $\R^m\setminus\delta^\bullet\Ss$. Consider the Nash diffeomorphism 
$$
H:\R^m\setminus\delta^\bullet\Ss\to M:=\{(x,t)\in\R^m\times\R:\ th(x)=1\}\subset\R^{m+1},\ x\mapsto\Big(x,\frac{1}{h(x)}\Big),
$$
whose inverse is the restriction to the Nash manifold $M$ (which is a closed subset of $\R^{m+1}$) of the projection $\pi:\R^m\times\R\to\R^m$ onto the first factor. Observe that $H(\Ss)$ is the closed semialgebraic set $\Tt:=M\cap(\cl(\Ss)\times\R)$, as required.
\end{proof}

\subsection{Whitney's semialgebraic topology}\label{wst}
Let $M\subset\R^m$ be a Nash manifold. For each integer $\mu\geq 0$ we denote by ${\mathcal S}^{\mu}(M):={\mathcal S}^{\mu}(M,\R)$ the ring of semialgebraic functions on $M$ of class $\Cont^{\mu}$. We equip ${\mathcal S}^{\mu}(M)$ with the ${\mathcal S}^{\mu}$ \textit{(strong) Whitney's semialgebraic topology} (${\mathcal S}^{\mu}$ topology, in short) \cite[\S II.1]{sh}. If $\mu\geq 1$, let $\xi_1,\ldots,\xi_s$ be semialgebraic tangent fields on $M$ that span the tangent bundle $TM$ of $M$. Recall that $TM:=\{(p,v)\in M\times\R^m:\ v\in T_pM\}$ where $T_pM$ is the tangent space to $M$ at $p$ for each $p\in M$. For every strictly positive continuous semialgebraic function $\veps:M\to\R$ we denote by ${\mathcal U}_{\veps}$ the set of all functions $g\in{\mathcal S}^{\mu}(M)$ such that
\begin{equation}\label{topology}
\begin{cases}
|g|<\veps&\text{if $\mu=0$,}\\
|g|<\veps\quad\text{and}\quad|\xi_{i_1}\cdots\xi_{i_\ell}(g)|<\veps\, 
 \text{ for }\, 1\le i_1,\dots,i_\ell\le s,\, 1\le\ell\le\mu&\text{if $\mu\geq1$}.
\end{cases}
\end{equation}
The sets ${\mathcal U}_\veps$ form a basis of open neighborhoods of the zero function for a topology in ${\mathcal S}^{\mu}(M)$ (recall that ${\mathcal S}^{\mu}(M)$ is a topological ring), which does not depend on the choice of the tangent fields if $\mu\geq1$. For each $f,g\in{\mathcal S}^{\mu}(M)$ if $f-g\in{\mathcal U}_\veps$, we say that $f$ is \textit{$\veps$-close} to $g$ (with respect to the ${\mathcal S}^{\mu}$ topology). When the meaning is clear from the context, to lighten the presentation we will omit the reference to the ${\mathcal S}^{\mu}$ topology, that is, we will only say $f$ is $\veps$-close to $g$. Let $\veps^*:M\to\R$ be a strictly positive continuous semialgebraic function such that $\veps^*\leq\veps$. Clearly, if $f$ is $\veps^*$-close to $g$, then $f$ is also $\veps$-close to $g$, because ${\mathcal U}_{\veps^*}\subset{\mathcal U}_{\veps}$. We say {\em $f$ is close to $g$} with respect to the ${\mathcal S}^{\mu}$ topology if there exists a small enough strictly positive continuous semialgebraic function $\veps:M\to\R$ such that $f-g\in{\mathcal U}_\veps$.

Let $N\subset\R^n$ be a Nash manifold. We denote by ${\mathcal S}^{\mu}(M,N)$ the space of ${\mathcal S}^\mu$ maps $f:M\to N$. On ${\mathcal S}^{\mu}(M,N)$ we consider the subspace topology given by the canonical inclusion in the following space endowed with the product topology \cite[II.1.3]{sh}
$$
{\mathcal S}^{\mu}(M,N)\hookrightarrow{\mathcal S}^{\mu}(M,\R^n)={\mathcal S}^{\mu}(M,\R)\times\overset{(n)}{\cdots}\times{\mathcal S}^{\mu}(M,\R),\ f\mapsto(f_1,\ldots,f_n).
$$
Roughly speaking, $g$ is close to $f$ when its components $g_k$ are close to the components $f_k$ of $f$. 

Let $\Ss\subset M$ be a locally compact semialgebraic set. We say that $f:\Ss\to\R^n$ is an ${\mathcal S}^{\mu}$ map on $\Ss$ if there exists an open semialgebraic neighborhood $U$ of $\Ss$ in $M$ and an ${\mathcal S}^{\mu}$ extension $F:U\to\R^n$ of $f$. As $M$ has a Nash tubular neighborhood in $\R^n$ with the corresponding Nash retraction \cite[Cor.8.9.5]{bcr}, the map $f$ is ${\mathcal S}^{\mu}$ if and only if it extends to an ${\mathcal S}^{\mu}$ map on an open semialgebraic neighborhood of $\Ss$ in $\R^m$ (as we have proposed in the Introduction), see \cite[Prop.2.C.3]{bfr}. If $n=1$, we denote by ${\mathcal S}^{\mu}(\Ss):={\mathcal S}^{\mu}(\Ss,\R)$ the ring of all ${\mathcal S}^{\mu}$ functions on $\Ss$. 

\subsubsection{Extension of ${\mathcal S}^{\mu}$ maps and invariance of the topology}\label{esmm}
If $\Ss\subset M$ is a closed semialgebraic set (which happens for instance if $M:=\R^n\setminus\delta^\bullet\Ss$) and $f\in{\mathcal S}^{\mu}(\Ss,\R^n)$, there exists an open semialgebraic neighborhood $W\subset M$ of $\Ss$ and an ${\mathcal S}^{\mu}$ map $F:W\to\R^n$ on $W$ such that $F|_\Ss=f$. Let $\sigma:M\to[0,1]$ be an ${\mathcal S}^{\mu}$ bump function such that $\sigma|_\Ss=1$ and $\sigma|_{M\setminus W}=0$. If we replace $F$ by $\sigma F$, we may assume that the ${\mathcal S}^{\mu}$ extension $F$ of $f$ is defined on $M$. We define the ${\mathcal S}^{\mu}$-topology of ${\mathcal S}^{\mu}(\Ss,\R^n)$ as the quotient topology induced by the restriction surjective homomorphism ${\mathcal S}^{\mu}(M,\R^n)\to{\mathcal S}^{\mu}(\Ss,\R^n),\ F\mapsto f$. {\em The ${\mathcal S}^{\mu}$-topology of ${\mathcal S}^{\mu}(\Ss,\R^n)$ does not depend on the choice of the Nash manifold $M$}. 
\begin{proof}
If $\Omega_1$ and $\Omega_2$ are two open semialgebraic neighborhoods of $\Ss$ in $\R^m$ such that $\Ss$ is closed both in $\Omega_1$ and $\Omega_2$ the reader proves straightforwardly that both ${\mathcal S}^{\mu}(\Omega_1,\R^n)$ and ${\mathcal S}^{\mu}(\Omega_2,\R^n)$ induce the same ${\mathcal S}^{\mu}$-topology on ${\mathcal S}^{\mu}(\Ss,\R^n)$ (use \eqref{topology}).

If $M\subset\R^m$ is a Nash manifold that contains $\Ss$ as a closed subset, then both $\Ss$ and $M$ are contained in the open semialgebraic set $\Omega:=\R^m\setminus\delta^\bullet M$ as closed subsets and by \cite[Prop.2.C.3]{bfr} both ${\mathcal S}^{\mu}(M,\R^n)$ and ${\mathcal S}^{\mu}(\Omega,\R^n)$ induce the same topology on ${\mathcal S}^{\mu}(\Ss,\R^n)$ using the chain of surjective restriction homomorphisms ${\mathcal S}^{\mu}(\Omega,\R^n)\to{\mathcal S}^{\mu}(M,\R^n)\to{\mathcal S}^{\mu}(\Ss,\R^n)$, as required. 
\end{proof}

If $\mu=0$ and $\Ss$ is locally compact in $\R^m$, there exist by \cite{dk} an open semialgebraic neighborhood $U$ of $\Ss$ in $\R^m$ and a continuous semialgebraic retraction $\nu:U\to\Ss$. Using an ${\mathcal S}^0$ bump function $\sigma:\R^m\setminus\delta^\bullet\Ss\to\R$ such that $\sigma|_\Ss=1$ and $\sigma|_{\R^m\setminus U}=0$, we have that $F:=(\sigma f)\circ\nu:\R^m\setminus\delta^\bullet\Ss\to\R^n$ defines an ${\mathcal S}^0$ map on $\R^m\setminus\delta^\bullet\Ss$ such that $F|_\Ss=f$. The space ${\mathcal S}^0(\Ss,\R^n)$ coincides with the space of continuous semialgebraic maps from $\Ss$ to $\R^n$ and there exists a restriction surjective map ${\mathcal S}^0(\R^m\setminus\delta^\bullet\Ss,\R^n)\to{\mathcal S}^0(\Ss,\R^n),\ F\mapsto F|_\Ss$ for each $n\geq1$. Let $\Tt\subset\R^n$ be any semialgebraic set and let $\Ss\subset\R^m$ be a locally compact semialgebraic set. We consider in ${\mathcal S}^{\mu}(\Ss,\Tt)$ the subspace topology induced by the canonical inclusion ${\mathcal S}^{\mu}(\Ss,\Tt)\hookrightarrow{\mathcal S}^{\mu}(\Ss,\R^n)$. 

\subsubsection{Continuity of composition homomorphisms}
With this ${\mathcal S}^{\mu}$-topology the composition on the left is a continuous map.

\begin{prop}[{\cite[Prop.2.D.1]{bfr}}]\label{leftcomp}
Let $\Ss\subset\R^m$ and $\Tt\subset\R^n$ be locally compact semialgebraic sets and let $\Tt'\subset\R^k$ be any semialgebraic set. Let $h:\Tt\to\Tt'$ be an ${\mathcal S}^{\mu}$ map. Then, the map $h_*:{\mathcal S}^{\mu}(\Ss,\Tt)\to{\mathcal S}^{\mu}(\Ss,\Tt'),\ f\mapsto h\circ f$ is continuous with respect the ${\mathcal S}^{\mu}$ topologies. 
\end{prop}

In general, the right composition is not continuous with respect to the ${\mathcal S}^{\mu}$ topologies, but if the involved ${\mathcal S}^{\mu}$ map is proper we have the following.

\begin{prop}\label{rightcomp}
Let $M\subset\R^m$ and $N\subset\R^n$ be Nash manifolds, $\Ss\subset M$ and $\Ss'\subset N$ closed semialgebraic sets and $\Tt\subset\R^k$ any semialgebraic set. Let $h:\Ss\to\Ss'$ be a proper ${\mathcal S}^{\mu}$ map. Then the map $h^*:{\mathcal S}^{\mu}(\Ss',\Tt)\to{\mathcal S}^{\mu}(\Ss,\Tt),\ f\mapsto f\circ h$ is continuous with respect the ${\mathcal S}^{\mu}$ topologies. 
\end{prop}

To prove the continuity of $h^*$ we will first show that the proper ${\mathcal S}^\mu$ map $h:\Ss\to\Ss'$ extends to a proper ${\mathcal S}^\mu$ map between some open semialgebraic neighborhoods $M'\subset M$ and $N'\subset N$ of $\Ss$ and $\Ss'$ respectively. We refer the reader to \cite[Lem.9.2]{bfr} for a result of similar nature, when dealing with local homeomorphisms.

\begin{lem}[Extension of proper ${\mathcal S}^{\mu}$ maps]\label{prepproper}
Let $M\subset\R^m$ and $N\subset\R^n$ be Nash manifolds and let $\Ss\subset M$ and $\Ss'\subset N$ be closed semialgebraic subsets. Let $h:\Ss\to\Ss'$ be a proper ${\mathcal S}^\mu$ map. Then there exist open semialgebraic neighborhoods $M'\subset M$ and $N'\subset N$ of $\Ss$ and $\Ss'$ respectively and a proper ${\mathcal S}^\mu$ map $H:M'\to N'$ that extends $h$.
\end{lem}
\begin{proof}
The proof is conducted in several steps.

\noindent{\sc Step 1.} Using suitable Nash embeddings of $M$ and $N$ in affine spaces (maybe of larger dimension, Lemma \ref{Mos}), we may assume that both $M$ and $N$ are closed Nash submanifolds of $\R^m$ and $\R^n$ and both $\Ss$ and $\Ss'$ are closed semialgebraic subsets of $M$ and $N$ respectively. As $\Ss$ and $\Ss'$ are closed semialgebraic subset of $M$ and $N$ respectively and $M$ and $N$ are closed subsets of $\R^m$ and $\R^n$ respectively, also $\Ss$ and $\Ss'$ are closed subsets of $\R^m$ and $\R^n$ respectively. Thus, we may assume $M=\R^m$ and $N=\R^n$. As $h:\Ss\to\Ss'$ is an ${\mathcal S}^\mu$ map, there exists an ${\mathcal S}^\mu$ map $H:\R^m\to\R^n$ such that $H|_{\Ss}=h$ (see \S\ref{esmm}). By \cite[Prop.I.4.5]{sh} there exists an ${\mathcal S}^\mu$ function $f$ on $\R^m$ such that $\Ss=\{f=0\}$. Define $H_1:\R^m\to\R^{n+1},\ x\mapsto(H(x),f(x))$, which is an ${\mathcal S}^\mu$ map such that $H_1^{-1}(\Ss'\times\{0\})=H^{-1}(\Ss')\cap\{f=0\}=\Ss$. 

For each integer $k\geq 1$ consider the inverse of the stereographic projection 
$$
\varphi_k:\R^k\to\sph^k\setminus\{p_k\},\ x:=(x_1,\ldots,x_k)\mapsto\Big(\frac{2x_1}{1+\|x\|^2},\ldots,\frac{2x_k}{1+\|x\|^2},\frac{-1+\|x\|^2}{1+\|x\|^2}\Big)
$$
with respect to the corresponding north pole $p_k:=(0,\ldots,0,1)\in\R^{k+1}$. Define $\Tt:=\varphi_m(\Ss)$ and $\Tt':=\varphi_{n+1}(\Ss')$, which are closed semialgebraic subsets of $\sph^m\setminus\{p_m\}$ and $\sph^{n+1}\setminus\{p_{n+1}\}$ respectively, and $G:=\varphi_{n+1}\circ H_1\circ\varphi_m^{-1}:\sph^m\setminus\{p_m\}\to\sph^{n+1}$, which is an ${\mathcal S}^\mu$ map, because so are the involved maps in the composition. In addition, the restriction $g:=G|_\Tt:=\varphi_{n+1}\circ(h,0)\circ\varphi_m^{-1}:\Tt\to\Tt'$ is proper, because so are the involved maps in the composition. As $H_1^{-1}(\Ss'\times\{0\})=\Ss$, we have $G^{-1}(\Tt')=\Tt$. As the maps $\varphi_m$ and $\varphi_n$ are Nash diffeomorphisms, so in particular proper maps, we are reduced to show: \textit{There exist an open semialgebraic neighborhood $M'\subset\sph^m\setminus\{p_m\}$ of $\Tt$ and an open semialgebraic neighborhood $N'\subset\sph^{n+1}\setminus\{p_{n+1}\}$ of $\Tt'$ such that the restriction $G|_{M'}:M'\to N'$ is a proper map.}
 
\noindent{\sc Step 2.} Define the map
$$
g_0:\Tt\sqcup\{p_m\}\to\Tt'\sqcup\{p_{n+1}\},\ z\mapsto
\begin{cases}
g(z)&\text{if $z\in\Tt$,}\\
p_{n+1}&\text{if $z=p_m$.}
\end{cases}
$$
As $\Tt$ is closed in $\sph^m\setminus\{p_m\}$, we deduce that $\Tt\sqcup\{p_m\}$ is closed in $\sph^m$ and consequently compact. Analogously, $\Tt'\sqcup\{p_{n+1}\}$ is closed in $\sph^{n+1}$, so it is a compact set. As $g$ is a proper map, $g_0$ is a continuous semialgebraic map between the compact semialgebraic sets $\Tt\sqcup\{p_m\}$ and $\Tt'\sqcup\{p_{n+1}\}$, so it is a continuous proper semialgebraic map. The graph $\Gamma(g_0)$ of $g_0$ is a closed semialgebraic subset of $\sph^m\times\sph^{n+1}$ such that 
\begin{equation}\label{chiusuragrafico0}
\Gamma(g_0)\cap((\{p_m\}\times\sph^{n+1})\cup(\sph^m\times\{p_{n+1}\}))=\{(p_m,p_{n+1})\},
\end{equation}
because $g_0(p_m)=p_{n+1}$ and $g_0^{-1}(p_{n+1})=\{p_m\}$. We deduce
\begin{equation}\label{chiusuragrafico}
\cl(\Gamma(g))\subset\Gamma(g_0)=\Gamma(g)\cup\{(p_m,p_{n+1})\}.
\end{equation}

Consider the graph $\Gamma(G)\subset\sph^m\times\sph^{n+1}$ of $G$ and its closure $C$ in $\sph^m\times\sph^{n+1}$, which is contained in $\sph^m\times\sph^{n+1}$. We claim: \textit{There exists an open semialgebraic neighborhood $W\subset\sph^m\setminus\{p_m\}$ of $\Tt$ such that the closure $C'$ of the graph $\Gamma(G|_{W})$ of $G|_{W}$ satisfies $C'\cap(\{p_m\}\times\sph^{n+1})\subset\{(p_m,p_{n+1})\}$}.

Consider the semialgebraic sets $\Gamma(g)$ and $C\cap(\{p_m\}\times\sph^{n+1})\setminus\{(p_m,p_{n+1})\}$, which are disjoint closed semialgebraic subsets of $(\sph^m\times\sph^{n+1})\setminus\{(p_m,p_{n+1})\}$, because by \eqref{chiusuragrafico} $\Gamma(g)=\cl(\Gamma(g))\setminus\{(p_m,p_{n+1})\}$ and by \eqref{chiusuragrafico0} 
\begin{multline*}
\Gamma(g)\cap (C\cap(\{p_m\}\times\sph^{n+1})\setminus\{(p_m,p_{n+1})\})\\
\subset\Gamma(g_0)\cap(((\{p_m\}\times\sph^{n+1})\cup(\sph^m\times\{p_{n+1}\}))\setminus\{(p_m,p_{n+1})\})=\varnothing.
\end{multline*}
Let $\Omega\subset(\sph^m\times\sph^{n+1})\setminus\{(p_m,p_{n+1})\}$ be an open semialgebraic neighborhood of $(\{p_m\}\times\sph^{n+1})\setminus\{(p_m,p_{n+1})\}$ such that $\cl(\Omega)\setminus\{(p_m,p_{n+1})\}$ does not meet $\Gamma(g)$, so $\cl(\Omega)\cap\Gamma(g)=\varnothing$, because $(p_m,p_{n+1})\not\in\Gamma(g)$. We deduce $(\{p_m\}\times\sph^{n+1})\setminus\Omega\subset\{(p_m,p_{n+1})\}$.

As $\Gamma(G)\setminus\cl(\Omega)$ is an open subset of $\Gamma(G)$ that contains $\Gamma(g)$ and does not meet $\cl(\Omega)$ and the projection $\pi:\sph^m\times\sph^{n+1}\to\sph^m$ onto the first factor induces a homeomorphism between $\Gamma(G)$ and $\sph^m\setminus\{p_m\}$, we deduce $\pi(\Gamma(G)\setminus\cl(\Omega))=\{x\in \sph^m\setminus\{p_m\}:\ (x,G(x))\not\in\cl(\Omega)\}$ is an open semialgebraic subset of $\sph^m\setminus\{p_m\}$ that contains $\pi(\Gamma(g))=\Tt$. 

As $\Tt$ is a closed subset of $\sph^m\setminus\{p_m\}$ and $\pi(\Gamma(G)\setminus\cl(\Omega))\subset\sph^m\setminus\{p_m\}$ is an open semialgebraic neighborhood of $\Tt$, there exists an open semialgebraic neighborhood $W\subset\pi(\Gamma(G)\setminus\cl(\Omega))$ of $\Tt$ such that $\cl(W)\setminus\{p_m\}\subset\pi(\Gamma(G)\setminus\cl(\Omega))\subset\sph^m\setminus\{p_m\}$. As $\Gamma(G|_{W})\subset\Gamma(G)\setminus\cl(\Omega)\subset C\setminus\Omega$, the closure $C'$ of $\Gamma(G|_{W})$ is contained in $C\setminus\Omega$, so 
$$
C'\cap(\{p_m\}\times\sph^{n+1})\subset(C\cap(\{p_m\}\times\sph^{n+1}))\setminus\Omega\subset\{(p_m,p_{n+1})\},
$$
as claimed.

\noindent{\sc Step 3.} Consider the map $G|_{W}:W\to\sph^{n+1}$ and observe that $(G|_{W})^{-1}(\Tt')=G^{-1}(\Tt')\cap W=\Tt$. The map $G|_{\cl(W)\setminus\{p_m\}}:\cl(W)\setminus\{p_m\}\to\sph^{n+1}$ is well-defined, because $G$ is defined on $\sph^m\setminus\{p_m\}$, and extends continuously to a semialgebraic map ${G^\bullet}:\cl(W)\cup\{p_m\}\to\sph^{n+1}$ such that ${G^\bullet}(p_m)=p_{n+1}$, because $C'\cap(\{p_m\}\times\sph^{n+1})\subset\{(p_m,p_{n+1})\}$. As $\cl(W)$ is a compact set, ${G^\bullet}$ is a proper map. Let $D:=(\cl(W)\cup\{p_m\})\setminus W$, which is a closed subset of $\cl(W)\cup\{p_m\}$ that contains $p_m$ and $\cl(W)\setminus D=W$. Thus, ${G^\bullet}(D)$ is a closed subset of $\sph^{n+1}$ that contains $p_{n+1}$. As $D\cap\Tt=\varnothing$ and ${G^\bullet}^{-1}(\Tt')=\Tt$, we have ${G^\bullet}(D)\cap\Tt'=\varnothing$, so $\Tt\cap {G^\bullet}^{-1}(G(D))={G^\bullet}^{-1}(\Tt')\cap {G^\bullet}^{-1}({G^\bullet}(D))=\varnothing$. Define $M':=\cl(W)\setminus {G^\bullet}^{-1}({G^\bullet}(D))=W\setminus {G^\bullet}^{-1}({G^\bullet}(D))$, which is an open semialgebraic neighborhood of $\Tt$ in $\sph^m\setminus\{p_m\}$, and $N':=\sph^{n+1}\setminus {G^\bullet}(D)$, which is an open semialgebraic neighborhood of $\Tt'$ in $\sph^{n+1}\setminus\{p_{n+1}\}$. We deduce ${G^\bullet}^{-1}(N')=\cl(W)\setminus {G^\bullet}^{-1}({G^\bullet}(D))=M'$ and $G|_{M'}=G^\bullet|_{M'}:M'\to N'$ is a proper ${\mathcal S}^\mu$ map that extends $g$, as required.
\end{proof}

We are ready to show Proposition \ref{rightcomp}.

\begin{proof}[Proof of Proposition \em \ref{rightcomp}]
By Lemma \ref{prepproper} there exists an open semialgebraic neighborhood $M'$ of $\Ss$ in $M$, an open semialgebraic neighborhood $N'$ of $\Ss'$ in $N$ and a proper ${\mathcal S}^\mu$ map $H:M'\to N'$ that extends $h$. We have the following commutative diagram:
$$
\xymatrix{
{\mathcal S}^{\mu}(\Ss',\Tt)\ar[d]^{\iota_1}\ar[r]^{h^*}&{\mathcal S}^{\mu}(\Ss,\Tt)\ar[d]^{\iota_2}\\
{\mathcal S}^{\mu}(\Ss',\R^k)\ar[r]^{h^*}&{\mathcal S}^{\mu}(\Ss,\R^k)\\
{\mathcal S}^{\mu}(N',\R^k)\ar[r]^{H^*}\ar[u]^{\rho_1}&{\mathcal S}^{\mu}(M',\R^k)\ar[u]^{\rho_2}
}
$$
where $\iota_1$ and $\iota_2$ are the canonical embeddings and $\rho_1$ and $\rho_2$ are the restriction quotient maps \cite[\S 2.C]{bfr}. As $H:M'\to N'$ is proper, the map $H^*$ is continuous \cite[II.1.5]{sh}, so the composition $\rho_2\circ H^*$ is continuous too. Thus, $h^*\circ \rho_1$ is continuous because $h^*\circ \rho_1=\rho_2\circ H^*$. But $\rho_1$ is a quotient map (because the target is an affine space), so $h^*:{\mathcal S}^{\mu}(\Ss',\R^k)\to{\mathcal S}^{\mu}(\Ss,\R^k)$ is continuous. As $h^*$ in the second row is continuous, $h^*\circ\iota_1$ is continuous as well. Consequently, also the map $\iota_2\circ h^*=h^*\circ\iota_1$ is continuous. As $\iota_2$ is an embedding, we conclude that $h^*:{\mathcal S}^{\mu}(\Ss',\Tt)\to{\mathcal S}^{\mu}(\Ss,\Tt)$ is continuous, as required. 
\end{proof}

\subsubsection{${\mathcal S}^\mu$ and Nash embeddings of Nash manifolds}\label{emb}
Recall that if $\mu\geq1$, an ${\mathcal S}^\mu$ (resp. Nash) map $\varphi:N_1\to N_2$ between Nash manifolds $N_1$ and $N_2$ is a {\em ${\mathcal S}^\mu$ (resp. Nash) embedding} if $\varphi:N_1\to\varphi(N_1)$ is an ${\mathcal S}^\mu$ (resp. Nash) diffeomorphism. Denote the set of ${\mathcal S}^\mu$ embeddings of $N_1$ in $N_2$ by ${\mathcal S}_{\rm emb}^\mu(N_1,N_2)$ and the set of Nash embeddings of $N_1$ in $N_2$ by ${\mathcal N}_{\rm emb}^\mu(N_1,N_2)$. Observe that ${\mathcal N}_{\rm emb}^\mu(N_1,N_2)={\mathcal S}_{\rm emb}^\mu(N_1,N_2)\cap{\mathcal N}(N_1,N_2)$. Proceeding as in the proof of \cite[Ch.2.Thm.1.4]{hi} one shows that ${\mathcal S}_{\rm emb}^\mu(N_1,N_2)$ is an open subset of ${\mathcal S}^\mu(N_1,N_2)$ with respect to the ${\mathcal S}^\mu$ topology. Consequently, ${\mathcal N}_{\rm emb}^\mu(N_1,N_2)$ is an open subset of ${\mathcal N}(N_1,N_2)$ with respect to the induced ${\mathcal S}^\mu$ topology.

\subsection{Trimmed Whitney's semialgebraic topology}\label{twst}
Let $\Ss\subset\R^m$ and $\Tt\subset\R^n$ be semialgebraic sets and fix an integer $\mu\geq1$. To approximate continuous semialgebraic homotopies $H:\Ss\times[0,1]\to\Tt$ such that $H(\cdot,t)$ is ${\mathcal S}^{\mu}$ for each $t\in[0,1]$ (but maybe the whole map $H$ is not an ${\mathcal S}^{\mu}$ map) we need to introduce a trimmed Whitney's semialgebraic topology ${\mathcal S}^{\mu}_{\t}$ that does not take care about derivatives with respect to $\t$ (see the statement of Theorem \ref{homotopy2}). We first recall the following fact.

\begin{lem}\label{st}
Let $\Ss\subset\R^m$ be a semialgebraic set and let $\veps:\Ss\times[0,1]\to\R$ be a strictly positive continuous semialgebraic function. Then there exists a strictly positive continuous semialgebraic function $\veps^*:\Ss\to\R$ such that $\veps^*(x)\leq\veps(x,t)$ for each $(x,t)\in\Ss\times[0,1]$.
\end{lem} 
\begin{proof}
The projection $\pi:\Ss\times[0,1]\to\Ss,\ (x,t)\mapsto x$ is open (because it is a projection), proper (because the inverse image of a compact set $K\subset\Ss$ is $\pi^{-1}(K)=K\times[0,1]$, which is a compact set) and a surjective continuous semialgebraic map. By \cite[Const.3.1]{fg} the function
$$
\veps^*:\Ss\to\R,\ (x,t)\to\veps^*(x):=\min_{t\in [0,1]}\{\veps(x,t)\}.
$$
is a continuous semialgebraic function smaller than or equal to $\veps$ and still strictly positive, as required.
\end{proof}

\subsubsection{Construction of the trimmed Whitney's semialgebraic topology}\label{ctwst}
Let $M\subset\R^m$ be a Nash manifold and let $\mu\geq0$ be an integer. If $\mu=0$, then ${\mathcal S}^0_\t(M\times[0,1],\R^n):={\mathcal S}^0(M\times[0,1],\R^n)$ endowed with the ${\mathcal S}^0$ topology. Assume next $\mu\geq1$ and let $\xi_1,\ldots,\xi_s$ be semialgebraic tangent fields on $M$ that span the tangent bundle of $M$. Define ${\mathcal S}^{\mu}_\t(M\times[0,1],\R^n)$ as the space of continuous semialgebraic maps $f:M\times[0,1]\to\R^n$ such that $\xi_{i_1}\cdots\xi_{i_\ell}(f):M\times[0,1]\to\R^n$ is a continuous semialgebraic map for each $1\le i_1,\dots,i_\ell\le s$ and each $1\le\ell\le\mu$

We equip ${\mathcal S}^{\mu}_\t(M\times[0,1],\R^n)$ with the ${\mathcal S}^{\mu}_{\t}$ \textit{trimmed (strong) Whitney's semialgebraic topology (${\mathcal S}^{\mu}_{\t}$ topology} in short) that we define as follows. Let $\veps:M\times[0,1]\to\R$ be a strictly positive continuous semialgebraic function. By Lemma \ref{st} we may assume that $\veps$ does not depend on the variable $t\in[0,1]$, that is, $\veps:M\to\R$ is a strictly positive continuous semialgebraic function. Denote~by ${\mathcal V}_{\veps}$ the set of all functions $g\in{\mathcal S}^{\mu}_\t(M\times[0,1],\R^n)$ such that
\begin{equation}\label{topologyt}
\begin{cases}
\|g\|<\veps&\text{if $\mu=0$,}\\
\|g\|<\veps\quad\text{and}\quad\|\xi_{i_1}\cdots\xi_{i_\ell}(g)\|<\veps\, 
 \text{ for }\, 1\le i_1,\dots,i_\ell\le s,\, 1\le\ell\le\mu&\text{if $\mu\geq1$}.
\end{cases}
\end{equation}
The sets ${\mathcal V}_\veps$ form a basis of open neighborhoods of the zero function for the ${\mathcal S}^{\mu}_{\t}$ topology in ${\mathcal S}^{\mu}_\t(M\times[0,1],\R^n)$. 

Let $\Ss\subset\R^m$ and $\Tt\subset\R^n$ be semialgebraic sets such that $\Ss$ is locally compact. A continuous semialgebraic map $f:\Ss\times[0,1]\to\Tt$ is an {\em ${\mathcal S}^\mu_\t$ map} if there exists an open semialgebraic neighborhood $U\subset\R^m$ of $\Ss$ in which $\Ss$ is closed and an ${\mathcal S}^\mu_\t$ extension $F_0:U\times[0,1]\to\R^n$ of $f$ (see Remark \ref{rtwst}(ii) below). Define $\delta^\bullet\Ss:=\cl(\Ss)\setminus\Ss$, which is a closed semialgebraic set, because $\Ss$ is locally compact. We have $U\subset\R^m\setminus\delta^\bullet\Ss$, because $\R^m\setminus\delta^\bullet\Ss$ is the largest open semialgebraic neighborhood of $\Ss$ in $\R^m$ that contains $\Ss$ as a closed subset. Let $\sigma:\R^m\setminus\delta^\bullet\Ss\to[0,1]$ be an ${\mathcal S}^\mu$ bump function such that $\sigma|_\Ss=1$ and $\sigma|_{(\R^m\setminus\delta^\bullet\Ss)\setminus U}=0$. Define $F:=(\sigma\times\id_{[0,1]}) F_0:(\R^m\setminus\delta^\bullet\Ss)\times[0,1]\to\R^n$, which is an ${\mathcal S}^\mu_\t$ extension of $f$. Denote~by ${\mathcal S}^{\mu}_\t(\Ss\times[0,1],\Tt)$ the space of ${\mathcal S}^\mu_\t$ maps between $\Ss\times[0,1]$ and $\Tt$. We consider the quotient topology in ${\mathcal S}^\mu_\t(\Ss\times[0,1],\R^n)$ induced by the canonical surjective restriction map ${\mathcal S}^\mu_\t((\R^m\setminus\delta^\bullet\Ss)\times[0,1],\R^n)\to{\mathcal S}^\mu_\t(\Ss\times[0,1],\R^n),\ F\mapsto F|_{\Ss\times[0,1]}$ and the ${\mathcal S}^\mu_\t$ topology of ${\mathcal S}^\mu_\t((\R^m\setminus\delta^\bullet\Ss)\times[0,1],\R^n)$. We endow ${\mathcal S}^{\mu}_\t(\Ss\times[0,1],\Tt)$ with the subspace topology induced by the canonical inclusion ${\mathcal S}^\mu_\t(\Ss\times[0,1],\Tt)\hookrightarrow{\mathcal S}^\mu_\t(\Ss\times[0,1],\R^n)$.

If $M\subset\R^m$ is a Nash manifold that contains $\Ss$ as a closed subset, the reader can check that the (quotient) ${\mathcal S}^\mu_\t$ topology of ${\mathcal S}^\mu_\t(\Ss\times[0,1],\R^n)$ induced by the canonical surjective restriction map ${\mathcal S}^\mu_\t(M\times[0,1],\R^n)\to{\mathcal S}^\mu_\t(\Ss\times[0,1],\R^n),\ F\mapsto F|_{\Ss\times[0,1]}$ and the ${\mathcal S}^\mu_\t$ topology of ${\mathcal S}^\mu_\t(\Ss\times[0,1],\R^n)$ is the same as before.

\begin{remarks}\label{rtwst}
(i) Due to the nature of the ${\mathcal S}^\mu_\t$ topology, the reader can check that the space ${\mathcal S}^\mu_\t(\Ss\times[0,1],\Tt)$ endowed with this topology satisfies the analogous results to the space ${\mathcal S}^\mu(\Ss\times[0,1],\Tt)$ endowed with the ${\mathcal S}^\mu$ topology. As it is quite straightforward, we do not include further details, but we will use them quite freely (only) in the proof of Theorem \ref{homotopy2}.

(ii) Let $M\subset\R^m$ be a Nash manifold and let $\Ss\subset M$ be a closed semialgebraic subset. Let $\Omega\subset M\times\R$ be an open neighborhood of $\Ss\times[0,1]$. We claim: {\em There exists an open semialgebraic neighborhood $U\subset M$ of $\Ss$ such that $U\times[0,1]\subset\Omega$.}

Define $U:=\{x\in M:\ \{x\}\times[0,1]\subset\Omega\}$, which contains $\Ss$. Pick a point $x\in U$. Then $\{x\}\times[0,1]\subset\Omega$, so the compact set $\{x\}\times[0,1]$ does not meet the closed subset $M\setminus\Omega$ of $M$. Let $\delta:=\frac{1}{2}\dist(\{x\}\times[0,1],M\setminus\Omega)$. Then $\Bb_m(x,\delta)\times[0,1]$ does not meet $M\setminus\Omega$, so $\Bb_m(x,\delta)\times[0,1]\subset\Omega$ and $\Bb_m(x,\delta)\subset U$. Thus, $U$ is an open semialgebraic neighborhood of $\Ss$ in $M$, as claimed.
\hfill$\sqbullet$
\end{remarks}

\section{Nash manifolds with corners}\label{s3}

In this section we present some concepts and results related to Nash manifolds with corners.

\subsection{Generalities}\label{angoliPre}
Let $\Ss\subset\R^n$ be a semialgebraic set. A point $x\in\Ss$ is \em smooth \em if there exists an open semialgebraic neighborhood $U\subset\R^n$ of $x$ such that $U\cap\Ss$ is a Nash manifold. The set $\Sth(\Ss)$ of smooth points of $\Ss$ is by \cite{st} a semialgebraic subset of $\R^n$ (and consequently a union of Nash submanifolds of $\R^n$ possibly of different dimension). 

Let $\Qq\subset\R^n$ be a Nash manifold with corners. The set of {\em internal points $\Int(\Qq)$ of} $\Qq$ is $\Int(\Qq):=\Sth(\Qq)$ and the {\em boundary $\partial\Qq$} of $\Qq$ is $\partial\Qq:=\Qq\setminus\Int(\Qq)$. Observe that the set of internal points $\Int(\Qq)$ and the boundary $\partial\Qq$ of $\Qq$ coincide with the set of internal points and boundary of $\Qq$ seen as a topological manifold with boundary. 

A Nash manifold with corners $\Qq\subset\R^n$ is locally compact, consequently: {\em $\Qq$ is a closed Nash submanifold with corners of the open semialgebraic set $\R^n\setminus(\cl(\Qq)\setminus\Qq)$.} In \cite{fgr} it is shown that $\Qq$ is a closed subset of an (affine) Nash manifold of its same dimension. Recall that a Nash subset $Y$ of a Nash manifold $M\subset\R^n$ {\em has only Nash normal-crossings in $M$} if for each point $y\in Y$ there exists an open semialgebraic neighborhood $U\subset M$ such that $Y\cap U$ is a {\em Nash normal-crossings divisor} of $U$, that is, a finite union of transversal Nash hypersurfaces of $U$.

\begin{thm}[{\cite[Thm.1.11]{fgr}}]\label{corners0}
Let $\Qq\subset\R^n$ be a $d$-dimensional Nash manifold with corners. There exists a $d$-dimensional Nash manifold $M\subset\R^n$ that contains $\Qq$ as a closed subset and satisfies:
\begin{itemize}
\item[\rm{(i)}] The Nash closure $Y$ of $\partial\Qq$ in $M$ has only Nash normal-crossings in $M$ and $\Qq\cap Y=\partial\Qq$.
\item[\rm{(ii)}] For every $x\in\partial\Qq$ the smallest analytic germ that contains the germ $\partial\Qq_x$ is $Y_x$.
\item[\rm{(iii)}] $M$ can be covered by finitely many open semialgebraic subsets $U_i$ (for $i=1,\ldots,r$) equipped with Nash diffeomorphisms $u_i:=(u_{i1},\dots,u_{id}):U_i\to\R^d$ such that:
$$\hspace{-3mm}
\begin{cases}
\text{$U_i\subset\Int(\Qq)$ or $U_i\cap\Qq=\varnothing$},&\text{\!if $U_i$ does not meet $\partial\Qq$,}\\ 
U_i\cap\Qq=\{u_{i1}\ge0,\dots,u_{ik_i}\ge0\},&\text{\!if $U_i$ meets $\partial\Qq$ (for a suitable $k_i\ge1$).} 
\end{cases}
$$
\end{itemize}
\end{thm} 

The Nash manifold $M$ is called a {\em Nash envelope of $\Qq$}. In general, it is not guaranteed that the Nash closure $Y$ of $\partial\Qq$ in $M$ is a Nash normal-crossings divisor of $M$ (that is, a finite union of transversal Nash hypersurfaces of $M$), as shown in the following example.

\begin{example}
The semialgebraic set $\Qq:=\{\x\geq 0,\y^2\leq\x^2-\x^4\}\subset\R^2$ is a Nash manifold with corners (see Figure \ref{teardrop}). Given any open semialgebraic neighborhood $U$ of $\Qq$ in $\R^2$ the Nash closure of $\partial\Qq$ in $U$ is not a Nash normal-crossings divisor of $U$.\hfill$\sqbullet$

\begin{figure}[!ht]
\begin{center}
\begin{tikzpicture}[scale=2.5]
\draw[thick,->] (-0.5, 0) -- (1.3, 0) ;
\draw[thick,->] (0, -0.8) -- (0, 0.8) ;
\draw [domain=0:1, draw=black, fill=gray!80, fill opacity=0.3,smooth,samples=1000]
(0, 0)
-- plot ({\x}, {sqrt(\x*\x-\x*\x*\x*\x)})
-- (1, 0);
\draw [domain=0:1, draw=black, fill=gray!80, fill opacity=0.3,samples=1000]
(0, 0)
-- plot ({\x},{ -sqrt(\x*\x-\x*\x*\x*\x)})
-- (1, 0);
\end{tikzpicture}
\end{center}
\caption{\small{The Nash manifold with corners $\Qq:=\{\x\geq 0,\y^2\leq\x^2-\x^4\}$.}}
\label{teardrop}
\end{figure}
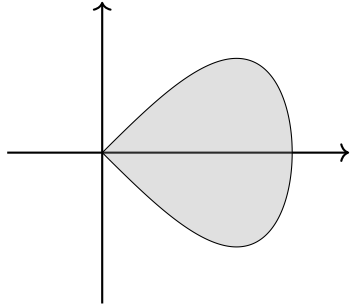
\end{example} 

We now define Nash manifolds with divisorial corners. 

\begin{defn}\label{nmwc}
A Nash manifold with corners $\Qq\subset\R^n$ is a Nash manifold with {\em divisorial corners} if there exists a Nash envelope $M\subset\R^n$ such that the Nash closure of $\partial\Qq$ in $M$ is a Nash normal-crossings divisor.\hfill$\sqbullet$
\end{defn}

A {\em facet of} a Nash manifold with corners $\Qq\subset\R^n$ is the (topological) closure in $\Qq$ of a connected component of $\Sth(\partial\Qq)$. As $\partial\Qq=\Qq\setminus\Sth(\Qq)$ is semialgebraic, the facets are semialgebraic and there are finitely many of them. The non-empty intersections of facets of $\Qq$ are the {\em faces of} $\Qq$. In \cite{fgr} the following characterization for Nash manifolds with divisorial corners is shown:

\begin{thm}[{\cite[Thm.1.12, Cor.6.5]{fgr}}]\label{divisorialPre}
Let $\Qq\subset\R^n$ be a $d$-dimensional Nash manifold with corners. The following assertions are equivalent:
\begin{itemize}
\item[\rm{(i)}] There exists a Nash envelope $M\subset\R^n$ where the Nash closure of $\partial\Qq$ is a Nash normal-crossings divisor.
\item[\rm{(ii)}] Every facet $\Ff$ of $\Qq$ is contained in a Nash manifold $X\subset\R^n$ of dimension $d-1$.
\item[\rm{(iii)}] The number of facets of $\Qq$ that contain every given point $x\in\partial\Qq$ coincides with the number of connected components of the germ $\Sth(\partial\Qq)_x$.
\item[\rm{(iv)}] All the facets of $\Qq$ are Nash manifolds with divisorial corners. 
\end{itemize}
If that is the case, the Nash manifold $M$ in {\rm(i)} can be chosen such that the Nash closure in $M$ of every facet $\Ff$ of $\Qq$ meets $\Qq$ exactly along $\Ff$.
\end{thm}

Properties (ii), (iii) and (iv) are intrinsic properties of $\Qq$ and do not depend on the Nash envelope $M$. The faces of a Nash manifold with divisorial corners are again Nash manifolds with divisorial corners. If a Nash envelope $M\subset\R^n$ of $\Qq$ satisfies the equivalent conditions of Theorem \ref{divisorialPre}, every open semialgebraic neighborhood $M'$ of $\Qq$ in $M$ is still a Nash envelope of $\Qq$ that satisfies such conditions. For the rest of this article (and also for the statements in the Introduction) we make the following assumption:

\begin{assumption}\label{nmwc0}
A Nash manifold with corners means a Nash manifold with divisorial corners.
\end{assumption}

We next recall a result from \cite[Thm.1.9]{cf2} that proves that Nash manifolds with corners admit Nash compactifications that are compact Nash manifolds with corners.

\begin{thm}\label{compactclosure}
Let $\Qq\subset\R^n$ be a Nash manifold with corners. Then there exists a Nash embedding ${\tt j}:\Qq\hookrightarrow\R^p$ for some $p\geq1$ such that $\cl(\Qq)\subset\R^p$ is a compact Nash manifold with corners.
\end{thm}

\section{Pushing Nash manifolds with corners inside its interior}\label{s4}

In this section we prove Theorem \ref{push}. By Theorem \ref{compactclosure} it is enough to prove Theorem \ref{push} when $\Qq$ is a compact Nash manifold with corners. 

\begin{proof}[Proof of Theorem {\em\ref{push}} in the compact case]
Let $\Qq\subset\R^n$ be a $d$-dimensional compact Nash manifold with corners and let $M\subset\R^n$ be a Nash envelope of $\Qq$ such that the Nash closure $X$ of $\partial\Qq$ in $M$ is a Nash normal-crossings divisor satisfying $\Qq\cap X=\partial\Qq$ (Theorem \ref{divisorialPre}). Let $X_1,\ldots,X_s$ be the irreducible components of $X$ (which are non-singular Nash hypersurfaces of $M$). Let $h_i:M\to\R$ be a Nash equation of $X_i$ such that $d_ph_i:T_pM\to\R$ is surjective for each $p\in X_i$ and each $i=1,\ldots,s$ and $\Qq=\{h_1\geq0,\ldots,h_s\geq0\}$ (see \cite[Cor.2.12]{cf3}). Let $(\Ee,\eta)$ be a Nash tubular neighborhood of $M$ in $\R^n$ (use \cite[Cor.8.9.5]{bcr}). The proof is conducted in two steps. 

\noindent{\sc Step 1. Construction of a suitable tangent Nash vector field.} 
We show: {\em There exists a Nash map $W:M\to\R^n$ such that $W$ is a tangent vector field on $M$ and $d_ph_j(W_p)>0$ for each $p\in\partial\Qq\cap\{h_j=0\}$ and each $j=1,\ldots,s$.}

By Theorem \ref{corners0} $M$ can be covered by finitely many open semialgebraic subsets $U_i$ (for $i=1,\ldots,r$) equipped with Nash diffeomorphisms $u_i:=(u_{i1},\dots,u_{id}):U_i\to\R^d$ such that:
$$\hspace{-3mm}
\begin{cases}
\text{$U_i\subset\Int(\Qq)$ or $U_i\cap\Qq=\varnothing$},&\text{\!if $U_i$ does not meet $\partial\Qq$,}\\ 
U_i\cap\Qq=\{u_{i1}\ge0,\dots,u_{ik_i}\ge0\},&\text{\!if $U_i$ meets $\partial\Qq$ (for a suitable $k_i\ge1$).} 
\end{cases}
$$
Consider the tangent Nash vector field $V_i$ on $U_i$ given by $(d_pu_i)^{-1}(1,\ldots,1)$ for each $p\in U_i$. Observe that $d_p(u_{ij})(V_{i,p})=1>0$ for each $i=1,\ldots,r$ and each $j=1,\ldots,d$. Let $\{\theta_i\}_{i=1}^r$ be an ${\mathcal S}^1$ partition of unity subordinated to $\{U_i\}_{i=1}^r$. Define the tangent ${\mathcal S}^1$ vector field $V:=\sum_{i=1}^r\theta_iV_i$ and we claim: {\em $d_ph_j(V_p)>0$ for each $p\in\partial\Qq\cap\{h_j=0\}$ and each $j=1,\ldots,s$.}

Fix $j=1,\ldots,s$ and let $p\in\partial\Qq\cap\{h_j=0\}$. If $p\not\in U_i$, then $\theta_i(p)=0$, because ${\rm supp}(\theta_i)\subset U_i$. After reordering the indices, we may assume that there exists $\ell\geq 1$ such that $p\in U_1\cap\cdots\cap U_\ell$ and $p\not\in U_{\ell+1}\cup\cdots\cup U_r$. Then $V_p=\sum_{i=1}^\ell\theta_i(p)V_{i,p}$ and $d_ph_j(V_p)=\sum_{i=1}^\ell\theta_i(p)d_ph_j(V_{i,p})$. Fix $i=1,\ldots,\ell$. As $h_j$ vanishes identically at $X_j$, after reordering the indices if necessary, we may assume that $h_j|_{U_i}=u_{i1}w_{ij}$ for some strictly positive Nash function $w_{ij}$ on $U_i$ (recall that $\Qq=\{h_1\geq0,\ldots,h_s\geq0\}$ and $h_i:M\to\R$ is an equation of $X_i$ such that $d_ph_i:T_pM\to\R$ is surjective for each $p\in X_i$ and each $i=1,\ldots,s$). We have $d_ph_j|_{U_i}=w_{ij}(p)d_pu_{i1}+u_{i1}(p)d_pw_{ij}=w_{ij}(p)d_pu_{i1}$, because $\{h_j|_{U_i}=0\}=\{u_{i1}=0\}$ and $h_j(p)=0$. Thus, $d_ph_j(V_{i,p})=w_{ij}(p)d_pu_{i1}(V_{i,p})=w_{ij}(p)>0$. As $\theta_i(p)\geq0$ for $i=1,\ldots,\ell$ and $\sum_{i=1}^\ell\theta_i(p)=1$, we conclude that
$$
d_ph_j(V_p)=\sum_{i=1}^\ell\theta_i(p)d_ph_j(V_{i,p})=\sum_{i=1}^\ell\theta_i(p)w_{ij}(p)>0,
$$
as claimed.

Denote~by $TM:=\{(x,v)\in M\times\R^n:\ v\in T_xM\}$ the tangent vector bundle of $M$, which is a Nash submanifold of $\R^{2n}$. Denote the projection onto the first factor by $\pi:TM\to M$. By \cite[Thm.II.4.1]{sh} there exist Nash approximations $W^*:M\to TM,\ x\mapsto(\phi(x),W^\bullet_x)$ of $V^*:M\to TM, \, x\mapsto (x,V_x)$ arbitrarily close to $V^*$ in the ${\mathcal S}^1$ topology. As $\pi\circ V^*=\id_M$ is an ${\mathcal S}^1$ diffeomorphism, we deduce by Proposition \ref{leftcomp} and \cite[Lem.II.1.7]{sh} that if $W^*$ is close enough to $V^*$, also $\pi\circ W^*=\phi$ is an ${\mathcal S}^1$ diffeomorphism and $\phi^{-1}$ is close to $\id_M$. By Proposition \ref{leftcomp} the ${\mathcal S}^1$ map $W:M\to \R^n,\, x\mapsto W^\bullet_{\phi^{-1}(x)}$ is a tangent vector field to $M$ close to $V$ in the ${\mathcal S}^1$ topology. As $d_ph_j(V_p)$ is strictly positive for each $j=1,\ldots,s$ and each $p$ belonging to the compact semialgebraic set $\partial\Qq\cap\{h_j=0\}$, we deduce by Proposition \ref{leftcomp} that, if $W$ is close enough to $V$, also $d_ph_j(W_p)$ is strictly positive for each $p\in\partial\Qq\cap\{h_j=0\}$ and each $j=1,\ldots,s$.

\noindent{\sc Step 2. Construction of the Nash map.} 
Define the Nash map 
$$
\varphi_0:M\times\R\to\R^n,\ (x,t)\mapsto x+t W_x. 
$$
As $\varphi_0(\Qq\times\{0\})=\Qq\subset M\subset\Ee$, we have $\Qq\times\{0\}\subset\varphi_0^{-1}(\Ee)$. Consider the closed semialgebraic set $C:=\R^n\setminus\varphi_0^{-1}(M)$ and let $\delta_0:=\frac{1}{2}\dist(\Qq\times\{0\},C)$. For each $0<\delta\leq\delta_0$ define 
$$
M_\delta:=\{x\in M:\ \dist(x,\Qq)<\delta\}, 
$$
which is an open semialgebraic subset of $M$ that contains $\Qq$. We claim: {\em There exists $0<\delta\leq\delta_0$ such that $M_\delta\times[-\delta,\delta]\subset\varphi_0^{-1}(\Ee)$.} 

Fix $0<\delta\leq\delta_0$ and let us first check: $(M_\delta\times[-\delta,\delta])\cap C=\varnothing$. 

Otherwise, there exists $(x,t)\in(M_\delta\times[-\delta,\delta])\cap C$. Then $\dist(x,\Qq)<\delta$, so there exists $x_0\in\Qq$ such that $d(x,x_0)<\delta$ (recall here that $\Qq$ is compact). Thus, $\|(x,t)-(x_0,0)\|^2=d(x,x_0)^2+t^2<2\delta^2$, so $\|(x,t)-(x_0,0)\|<\sqrt{2}\delta<2\delta$. As $(x,t)\in C$ and $(x_0,0)\in\Qq\times\{0\}$, then
$$
2\delta_0=\dist(\Qq\times\{0\},C)\leq \|(x,t)-(x_0,0)\|<2\delta\leq2\delta_0,
$$
which is a contradiction, so $(M_\delta\times[-\delta,\delta])\cap C=\varnothing$. 

As $\Qq$ is compact, there exists $N>0$ such that $W_x\in\Bb_n(0,N)$ for each $x\in\Qq$. Let $\rho:M_{\delta_0}\times\Bb_n(0,N+\delta_0)\times[-\delta_0,\delta_0]\to\R^n,\ (x,y,t)\mapsto x+ty$. As $\rho(\Qq\times\ol{\Bb}_n(0,N)\times\{0\})=\Qq\subset M\subset\Ee$ and $\Qq$ is compact, there exists $0<\delta\leq\delta_0$ such that $\rho(M_\delta\times\Bb_n(0,N)\times[-\delta,\delta])\subset\Ee$. We have 
$$
\varphi_0(M_\delta\times[-\delta,\delta])\subset\rho(M_\delta\times\Bb_n(0,N)\times[-\delta,\delta])\subset\Ee,
$$
because $W_x\in\Bb_n(0,N)$ for each $x\in\Qq$. Consequently,
$$
M_\delta\times[-\delta,\delta]\subset\varphi_0^{-1}(\varphi_0(M_\delta\times[-\delta,\delta]))\subset\varphi_0^{-1}(\Ee),
$$
as claimed.

Fix $0<\delta\leq\delta_0$ satisfying that $M_\delta\times[-\delta,\delta]\subset\varphi_0^{-1}(\Ee)$ and define $M_0:=M_\delta$. Recall that $(\Ee,\eta)$ is a Nash tubular neighborhood of $M$ in $\R^n$ and for each $j=1,\ldots,s$ consider the Nash functions $h_j^*:=h_j\circ\eta:\Ee\to\R$ and 
$$
h^\bullet_j:=h_j^*\circ\rho:M_0\times\Bb_n(0,N)\times[-\delta,\delta]\to \R,\ (x,y,t)\mapsto h_j^*(x+ty).
$$ 
Let $\pi_1:\R^n\times\R^n\times\R\to\R^n,\ (x,y,t)\mapsto x$ be the projection onto the first factor and let $\pi_2:\R^n\times\R^n\times\R\to\R^n\times\R^n,\ (x,y,t)\mapsto(x,y)$ be the projection onto the first two factors.

As $(h^\bullet_j-h^*_j\circ\pi_1)(x,y,0)=h^*_j(x)-h^*_j(x)=0$ for each $(x,y)\in M_0\times\Bb_n(0,N)$, the Nash function $h^\bullet_j-h_j^*\circ\pi_1$ vanishes on the hyperplane $\{\t=0\}$ of $M_0\times\Bb_n(0,N)\times[-\delta,\delta]$. Thus, there exists a Nash function $H_j:M_0\times\Bb_n(0,N)\times[-\delta,\delta]\to\R$ such that $h^\bullet_j-h^*_j\circ\pi_1=\t H_j$. Consider the Nash function $H_{j0}:M_0\times\Bb_n(0,N)\to\R,\ (x,y)\mapsto H_j(x,y,0)$. As $(H_j-H_{j0}\circ\pi_2)(x,y,0)=H_{j0}(x,y)-H_{j0}(x,y)=0$ for each $(x,y)\in M_0\times\Bb_n(0,N)$, the Nash function $H_j-H_{j0}\circ\pi_2$ vanishes on the hyperplane $\{\t=0\}$ of $M_0\times\Bb_n(0,N)\times[-\delta,\delta]$. Thus, there exists a Nash function $G_j:M_0\times\Bb_n(0,N)\times[-\delta,\delta]\to\R$ such that $H_j-H_{j0}\circ\pi_2=\t G_j$. Consequently, 
\begin{multline}\label{taylor0}
h_j^*(x+ty)=h^\bullet_j(x,y,t)=(h_j^*\circ\pi_1)(x,y,t)+tH_{j}(x,y,t)\\
=h_j^*(x)+t((H_{j0}\circ\pi_2)(x,y,t)+tG_j(x,y,t))=h_j^*(x)+tH_{j0}(x,y)+t^2G_j(x,y,t)
\end{multline}
for each $(x,y,t)\in M_0\times\Bb_n(0,N)\times[-\delta,\delta]$. 

Consider the restriction $\varphi:=\varphi_0|_{M_0\times[-\delta,\delta]}:M_0\times[-\delta,\delta]\to\Ee,\ (x,t)\mapsto x+t W_x$ and observe that $\eta\circ\varphi:M_0\times[-\delta,\delta]\to M,\ (x,t)\mapsto\eta(x+tW_x)$. We deduce from \eqref{taylor0}
\begin{equation}\label{ProvaTaylor}
(h_j^*\circ\varphi)(x,t)=h_j^*(x+tW_x)=h_j^*(x)+tH_{j0}(x,W_x)+t^2G_j(x,W_x,t)
\end{equation}
for each $(x,t)\in M_0\times[-\delta,\delta]$. Thus,
$$
H_{j0}(x,W_x)=\frac{\partial(h_j^*\circ\varphi)}{\partial\t}(x,0)=d_xh_j^*(W_x)=d_xh_j(W_x)
$$
for each $x\in M_0$, because $W_x$ is tangent to $M_0\subset M$ at $x$ and $h^*_j|_{M_0}=(h_j\circ\eta)|_{M_0}=h_j|_{M_0}$. As $h_j^*(x)=h_j(x)$ for each $x\in M_0$, we have by \eqref{ProvaTaylor}
\begin{equation}\label{Taylor2}
(h_j^*\circ\varphi)(x,t)=h_j(x)+td_xh_j(W_x)+t^2G_j(x,W_x,t)
\end{equation}
for each $(x,t)\in M_0\times[-\delta,\delta]$.

Fix $j=1,\ldots,s$. We claim: {\em For each $y\in\Qq$ there exist an open semialgebraic neighborhood $U^y\subset M_0$ of $y$ and $0<\veps_y<\delta$ such that $(h_j^*\circ\varphi)(x,t)>0$ for each $(x,t)\in(U^y\cap\Qq)\times(0,\veps_y]$.}

If $y\in\Qq\setminus\{h_j=0\}$, then $h_j(y)>0$. As $M_0$ is locally compact, there exists an open semialgebraic neighborhood $U^y\subset M_0$ of $y$ such that $\cl(U^y)$ is compact and $h_j|_{\cl(U^y)}>0$. Denote~by
\begin{align*}
n_{0,y}&:=\min\{h_j(x):\ x\in\cl(U^y)\}>0,\\ 
N_{1,y}&:=\max\{|d_xh_j(W_x)|:\ x\in\cl(U^y)\},\\
N_{2,y}&:=\max\{|G_j(x,W_x,t)|:\ (x,t)\in\cl(U^y)\times[0,\delta]\}.
\end{align*}
By \eqref{Taylor2} we have
$$
(h_j^*\circ\varphi)(x,t)=h_j(x)+td_xh_j(W_x)+t^2G_j(x,W_x,t)\geq n_{0,y}-tN_{1,y}-t^2N_{2,y}
$$
and there exists $0<\veps_y<\delta$ such that if $t\in(0,\veps_y]$, then $n_{0,y}-tN_{1,y}-t^2N_{2,y}>0$. Thus, $(h_j^*\circ\varphi)(x,t)>0$ for each $(x,t)\in(U^y\cap\Qq)\times(0,\veps_y]$.

If $y\in\Qq\cap\{h_j=0\}$, then $h_j(y)=0$ and $d_yh_j(W_y)>0$. Let $U^y\subset M_0$ be an open semialgebraic neighborhood of $y$ such that $\cl(U^y)$ is compact and $d_xh_j(W_x)>0$ for each $x\in\cl(U^y)$. Denote~by
\begin{align*}
n_{1,y}&:=\min\{d_xh_j(W_x):\ x\in\cl(U^y)\}>0\\
N_{2,y}&:=\max\{|G_j(x,W_x,t)|:\ (x,t)\in\cl(U^y)\times[0,1]\}.
\end{align*}
By \eqref{Taylor2} we deduce
$$
(h_j^*\circ\varphi)(x,t)=h_j(x)+td_xh_j(W_x)+t^2G_j(x,W_x,t)\geq t(n_{1,y}-tN_{2,y})
$$
and there exists $0<\veps_y<\delta$ such that if $t\in(0,\veps_y]$, then $n_{1,y}-tN_{2,y}>0$. Thus, $(h_j^*\circ\varphi)(x,t)>0$ for each $(x,t)\in (U^y\cap\Qq)\times(0,\veps_y]$, as claimed.

As $\Qq$ is compact, there exist $y_1,\ldots,y_\ell\in\Qq$ such that $\Qq\subset U^{y_1}\cup\cdots\cup U^{y_\ell}$. Define $\veps_j:=\min\{\veps_{y_1},\ldots,\veps_{y_\ell}\}<\delta$. We have: {\em $(h_j^*\circ\varphi)(x,t)>0$ for each $(x,t)\in\Qq\times(0,\veps_j]$}. As the previous property holds for $j=1,\ldots,s$, we define $0<\veps:=\min\{\veps_1,\ldots,\veps_s\}<\delta$. We have $(h_j\circ\eta\circ\varphi)(x,t)=(h_j^*\circ\varphi)(x,t)>0$ for each $(x,t)\in\Qq\times(0,\veps]$ and each $j=1,\ldots,s$. We conclude $(\eta\circ\varphi)(x,t)\in\Int(\Qq)=\{h_1>0,\ldots,h_s>0\}$ for each $(x,t)\in\Qq\times(0,\veps]$. In addition, $\eta\circ\varphi(x,0)=x$ for each $x\in M_0$. Define the Nash map
$$
\Sigma:=\eta\circ\varphi:M_0\times[-1,1]\to M,\ (x,t)\mapsto(\eta\circ\varphi)(x,\veps t).
$$ 
The Nash map $\sigma:=\Sigma|_{\Qq\times[0,1]}:\Qq\times[0,1]\to\Qq$ satisfies $\sigma(\Qq\times(0,1])=\Int(\Qq)$ and $\sigma|_{\Qq\times\{0\}}=\id_\Qq$, as required.
\end{proof}

\section{Pushing Nash manifolds with corners with control of derivatives}\label{s5}

In this section we improve Theorem \ref{push} showing that it is possible to push a Nash manifold with corners $\Qq\subset\R^n$ inside its interior $\Int(\Qq)$ using a one-parameter family of Nash diffeomorphisms (onto their images) arbitrarily close to the identity map with respect to the ${\mathcal S}^{\mu}$ topology for each integer $\mu\geq 0$. 

\subsection{Strictly positive Nash functions close to zero}

We show that there exist strictly positive Nash functions $\delta:\Qq\to(0,1)$ arbitrarily close to zero with respect to the ${\mathcal S}^{\mu}$ topology.

\begin{lem}\label{zeroNash}
Let $M\subset\R^n$ be a Nash manifold and $U\subset M$ an open semialgebraic set. Then, for each integer $\mu\geq 0$ there exist strictly positive Nash functions $h:U\to(0,1)$ arbitrarily close to zero with respect to the ${\mathcal S}^{\mu}$ topology. 
\end{lem}

Before proving the previous lemma we need some preliminary results. Given a multi-index $\alpha:=(\alpha_1,\ldots,\alpha_n)\in \N^n$ we denote by $|\alpha|=\alpha_1+\cdots+ \alpha_n$ its \textit{length}. The \textit{factorial of} $\alpha$ is defined as $\alpha!:=\alpha_1!\cdots\alpha_n!$, while for each $x:=(x_1,\ldots,x_n)\in\R^n$ the $\alpha^{th}$-\textit{power of} $x$ is defined as
$$
x^{\alpha}=x_1^{\alpha_1}\cdots x_n^{\alpha_n}.
$$
For each $\beta:=(\beta_1,\ldots,\beta_n)\in \mathbb{N}^n$ the \textit{sum of $\alpha$ and $\beta$} is defined component-wise, that is
$$
\alpha+\beta:=(\alpha_1+\beta_1,\ldots,\alpha_n+\beta_n).
$$
Moreover, we write $\beta\leq\alpha$ if $\beta_i\leq \alpha_i$ for each $i=1,\ldots,n$.

We start with the following two well-known elementary results that we include here for the sake of completeness. 

\begin{lem}\label{derlem2}
For each $\alpha:=(\alpha_1,\ldots,\alpha_n)\in\mathbb{N}^n$ and each integer $m\geq 1$ we have
\begin{equation}\label{derlem2f}
\sum_{\beta_1+\cdots+\beta_m=\alpha}\frac{\alpha!}{\beta_1!\cdots\beta_m!}=m^{|\alpha|}
\end{equation}
\end{lem}
\begin{proof}
Let $m\geq 1$ be an integer. Denote~by $x_i:=(x_{i1},\ldots,x_{in})$. We consider the following preliminary formula
\begin{equation}\label{derlem2f1}
(x_1+\cdots+x_m)^{\alpha}=\sum_{\beta_1+\cdots+ \beta_m=\alpha}\frac{\alpha!}{\beta_1!\cdots\beta_m!}\prod_{i=1}^m x_i^{\beta_i}
\end{equation}
The case $n=1$ follows from \cite[Prop.3.10]{ai}. We proceed by induction on $n$. Denote~by $\alpha':=(\alpha_1,\ldots,\alpha_{n-1})$ and $x_i':=(x_{i1},\ldots,x_{i,n-1})$ for $n\geq2$ and suppose
$$
(x_1'+\cdots+x_m')^{\alpha'}=\sum_{\beta_1'+\cdots+\beta_m'=\alpha'}\frac{\alpha'!}{\beta_1'!\cdots\beta_m'!}\prod_{i=1}^m x_i'^{\beta_i'}
$$ 
where $\beta_i':=(\beta_{i1},\ldots,\beta_{i,n-1})$. Then
\begin{multline*}
(x_1+\cdots+x_m)^{\alpha}=(x_1'+\cdots+x_m')^{\alpha'}(x_{1n}+\cdots+x_{mn})^{\alpha_n}\\
=\Big(\sum_{\beta_1'+\cdots+\beta_m'=\alpha'}\frac{\alpha'!}{\beta_1'!\cdots\beta_m'!}\prod_{i=1}^m x_i'^{\beta_i'}\Big)\Big(\sum_{\beta_{1n}+\cdots+\beta_{mn}=\alpha_n}\frac{\alpha_n!}{\beta_{1n}!\cdots\beta_{mn}!}\prod_{i=1}^m x_{in}^{\beta_{in}}\Big)\\
=\sum_{\beta_1+\cdots+ \beta_m=\alpha}\frac{\alpha!}{\beta_1!\cdots\beta_m!}\prod_{i=1}^m x_i^{\beta_i}
\end{multline*}
where $\beta_i:=(\beta_i',\beta_{in})$.

Substituting $x_i=(1,\ldots,1)$ in \eqref{derlem2f1} for each $i=1,\ldots,m$, we deduce
\begin{equation}\label{K}
\sum_{\beta_1+\cdots+ \beta_m=\alpha}\frac{\alpha!}{\beta_1!\cdots\beta_m!}=m^{\alpha_1}\cdots m^{\alpha_n}=m^{|\alpha|}
\end{equation}
for each $\alpha\in\mathbb{N}^n$ and each integer $m\geq 1$, as required.
\end{proof}

Let $U\subset\R^n$ be an open subset and let $f:U\to\R$ be a function of class $\Cont^{\mu}$ for some $\mu\geq 1$. For each $\alpha\in \mathbb{N}^{n}$ such that $|\alpha|\leq\mu$ we denote by
$$
D^{(\alpha)} f:=\frac{\partial^{|\alpha|}}{\partial \x_1^{\alpha_1}\ldots\partial \x_n^{\alpha_n}}f,
$$
the partial derivative of $f$ of order $\alpha$, where $D^{(\alpha)}f=f$ if $|\alpha|=0$. 

\begin{lem}[Leibniz rule for powers]\label{derlem1}
Let $U\subset\R^n$ be an open subset and $\mu\geq 1$ an integer. Let $f:U\to\R$ be a function of class $\Cont^{\mu}$. Then for each $\alpha\in\mathbb{N}^n$ such that $|\alpha|\leq\mu$ and each integer $m\geq 1$, it holds
$$
D^{(\alpha)} f^m=\displaystyle\sum_{\beta_1+\cdots+ \beta_m=\alpha}\frac{\alpha!}{\beta_1!\cdots\beta_m!}\prod_{1\leq r\leq m} D^{(\beta_r)}f.
$$
\end{lem}
\begin{proof}
Using generalized Leibniz's formula \cite[\S6.4]{sa} we have 
\begin{equation}\label{Leibniz}
D^{(\alpha)}(gh)=\sum_{\beta\leq \alpha}\frac{\alpha!}{\beta!(\alpha-\beta)!}D^{(\beta)}gD^{(\alpha-\beta)}h
\end{equation}
for each $1\leq|\alpha|\leq\mu$ and each pair of functions $g,h:U\to\R$ of class $\Cont^{\mu}$. As the statement is clear when $m=1$, we may assume $m\geq 2$. Using \eqref{Leibniz} and an inductive argument we have
\begin{align*}
D^{(\alpha)}f^m&=\sum_{\beta_1\leq \alpha}\frac{\alpha!}{\beta_1!(\alpha-\beta_1)!}(D^{(\beta_1)}f)(D^{(\alpha-\beta_1)}f^{m-1})\\
&=\sum_{\beta_1\leq \alpha}\frac{\alpha!}{\beta_1!(\alpha-\beta_1)!}(D^{(\beta_1)}f)\displaystyle\sum_{\beta_2+\cdots+ \beta_m=\alpha-\beta_1}\frac{(\alpha-\beta_1)!}{\beta_2!\cdots \beta_{m}!}\prod_{2\leq r\leq m} D^{(\beta_r)}f\\
&=\sum_{\beta_1+\cdots+ \beta_m=\alpha}\frac{\alpha!}{\beta_1!\cdots\beta_m!}\prod_{1\leq r\leq m} D^{(\beta_r)}f,
\end{align*}
for each $m\geq 2$, as required.
\end{proof}

We next prove the following \L ojasiewicz's style formula.

\begin{lem}\label{derlem3}
Let $\mu\geq 1$ be an integer. Let $U\subset\R^n$ be an open set such that $\cl(U)$ is compact and $f:\cl(U)\to\R$ a function of class $\Cont^{\mu}$ such that $|f(x)|<1$ for each $x\in \cl(U)$. Then there exists an integer $M>\mu$ such that for each integer $N\geq M$ it holds
$$
|D^{(\alpha)}f^N(x)|<|f(x)|
$$
for each $x\in U$ and each $\alpha\in\N^n$ such that $|\alpha|\leq\mu$.
\end{lem}
\begin{proof}
Let $\alpha\in\mathbb{N}^n$ be such that $|\alpha|\leq\mu$. Let $m\geq 1$ be such that $m>|\alpha|$ and let $\beta_1,\ldots,\beta_m\in\mathbb{N}^n$ such that $\beta_1+\cdots+ \beta_m=\alpha$. Then there exist at least $m-|\alpha|$ indices $r\in\{1,\ldots,m\}$ such that $|\beta_r|=0$ and at most $|\alpha|$ indices $r\in\{1,\ldots,m\}$ such that $|\beta_r|\neq0$ . Define
$$
C:=1+\max_{|\alpha|\leq\mu}\{\max_{x\in \cl(U)}\{|D^{(\alpha)}f(x)|\}\}\geq1\quad\text{and}\quad L:=\max_{x\in \cl(U)}\{|f(x)|\}<1.
$$
We have $L<1$, because $\cl(U)$ is compact and $|f(x)|<1$ for each $x\in \cl(U)$. By Lemmas \ref{derlem2} and \ref{derlem1} and as $|f(x)|<1$ for each $x\in U$, we deduce
\begin{align}
\begin{split}\label{multinomial2}
|D^{(\alpha)}f^m(x)|&\leq\sum_{\beta_1+\cdots+\beta_m=\alpha}\frac{\alpha!}{\beta_1!\cdots\beta_m!}\prod_{1\leq r\leq m}|D^{(\beta_r)}f(x)|\\
&=\sum_{\beta_1+\cdots+ \beta_m=\alpha}\frac{\alpha!}{\beta_1!\cdots\beta_m!}\prod_{\substack{1\leq r\leq m\\ 
|\beta_r|\neq0}}|D^{(\beta_r)}f(x)|\prod_{\substack{1\leq r\leq m\\ 
|\beta_r|=0}}|f(x)|\\
&\leq \sum_{\beta_1+\cdots+ \beta_m=\alpha}\frac{\alpha!}{\beta_1!\cdots\beta_m!}|f(x)|^{m-|\alpha|}\prod_{\substack{1\leq r\leq m\\
|\beta_r|\neq0}}C\\
&\leq C^{|\alpha|}|f(x)|^{m-|\alpha|}\sum_{\beta_1+\cdots+ \beta_m=\alpha}\frac{\alpha!}{\beta_1!\cdots\beta_m!} = C^{|\alpha|}m^{|\alpha|}|f(x)|^{m-|\alpha|}\\
&\leq (Cm)^\mu|f(x)|^{m-\mu}\leq(Cm)^\mu L^{m-\mu-1}|f(x)|
\end{split}
\end{align}
for each $x\in U$, each $\alpha\in\mathbb{N}^n$ such that $|\alpha|\leq\mu$ and each integer $m>\mu$. We claim: {\em there exists an integer $M>\mu$ such that 
\begin{equation}\label{bound}
(CN)^{\mu}L^{N-\mu-1}<1.
\end{equation}
for each integer $N\geq M$.} 

Consider the real function $g:(0,+\infty)\to\R$ given by
\begin{multline*}
g(x):=\log((Cx)^\mu L^{x-\mu-1})=\mu\log(C)+\mu\log(x)+(x-\mu-1)\log(L)\\
=\mu(\log(C)-\log(L))-\log(L)+\Big(\mu\frac{\log(x)}{x}+\log(L)\Big)x
\end{multline*}
As $\mu(\log(C)-\log(L))-\log(L)$ is a positive constant (because $C>1$ and $0<L<1$), $\lim_{x\to+\infty}\frac{\log(x)}{x}=0$ and $\log(L)<0$, there exists an integer $M>\mu$ such that if $x\geq M$, then $g(x)<0$. Thus, $\exp(g(x))=(Cx)^\mu L^{x-\mu-1}<1$ for each $x\geq M$. Consequently, the formula \eqref{bound} holds for each integer $N\geq M$.

By \eqref{multinomial2} and \eqref{bound} we deduce 
$$
|D^{(\alpha)}f^N(x)|\leq(CN)^\mu L^{N-\mu-1}|f(x)|<|f(x)|
$$
for each $x\in U$ and each $\alpha\in \mathbb{N}^n$ such that $|\alpha|\leq\mu$, as required.
\end{proof}
 
Next, we show that there exist strictly positive ${\mathcal S}^{\mu}$ functions arbitrarily close to zero with respect to the ${\mathcal S}^{\mu}$ topology.

\begin{lem}\label{derlem4}
Let $M\subset\R^n$ be a Nash manifold and $U\subset M$ an open semialgebraic set. Then, for each integer $\mu\geq 0$ there exist strictly positive ${\mathcal S}^{\mu}$ functions $h:U\to(0,1)$ arbitrarily close to zero with respect to the ${\mathcal S}^{\mu}$ topology. 
\end{lem}
\begin{proof}
Let $\veps:U\to\R$ be a strictly positive continuous semialgebraic function. As $U$ is open in $M$ and $M$ is a Nash submanifold of $\R^n$, $U$ itself is a Nash submanifold of $\R^n$. After replacing $U$ by a Nash tubular neighborhood $(\Omega,\nu)$ of $U$ in $\R^n$ and $\veps$ by $\veps\circ\nu$, we may assume: \textit{$U$ is an open semialgebraic subset of $\R^n$.} Let $\Bb_n:=\{x\in\R^n:\ \|x\|<1\}$ be the open ball centered at the origin and radius $1$ and consider
$$
\Phi:\R^n\to\Bb_n,\ x\mapsto\frac{x}{\sqrt{1+\|x\|^2}},
$$
which is a Nash diffeomorphism. After replacing $U$ by $\Phi(U)$ and $\veps$ by $\veps\circ \Phi^{-1}$, we may assume that $U$ is a bounded set, so: \textit{$\cl(U)$ is a compact semialgebraic set.}

The exterior boundary $\delta^\bullet U:=\cl(U)\setminus U$ is a closed semialgebraic subset of $\R^n$. By \cite[Prop.I.4.5]{sh} there exists an ${\mathcal S}^{\mu}$ function $f:\R^n\to\R$ such that $\delta^\bullet U=\{f=0\}$. Define
$$
g:=\frac{f}{2(1+f^2)}:\R^n\to\R,
$$
which is an ${\mathcal S}^{\mu}$ function such that:
\begin{itemize}
\item $\delta^\bullet U=\{g=0\}$,
\item $|g(x)|< \tfrac{1}{2}$ for each $x\in\R^n$. 
\end{itemize}
We replace $g$ with its restriction $g|_{\cl(U)}:\cl(U)\to\R$. By \cite[Prop.2.6.4]{bcr} there exists an integer $N>0$ such that the semialgebraic function $\veps^{*}:\cl(U)\to\R$ defined as 
$$
\veps^*(x):=
\begin{cases}
g^{2N}\veps&\text{if $x\in U$},\\
0&\text{if $x\in\delta^\bullet U$}
\end{cases}
$$
is continuous. Observe that $\veps^*$ is strictly positive on $U$ because $g\neq 0$ on $U$. As $|g|<\tfrac{1}{2}$, we deduce $\veps^*(x)<\veps(x)$ for each $x\in U$. After replacing $\veps$ by $\veps^*$, we may assume: \textit{The semialgebraic function $\veps$ extends continuously on $\cl(U)$ by setting $\veps(0)=0$ on $\delta^\bullet U$.}

As $\{\veps=0\}=\{g=0\}$, there exist by \L{}ojasiewicz's inequality \cite[Cor.2.6.7]{bcr} an integer $N_0\geq 0$ and a real number $C>0$ such that $|g^{N_0}(x)|\leq C\veps(x)$ for each $x\in\cl(U)$. 

Let $N_1>0$ be an integer such that $\tfrac{C}{2^{N_1}}<1$. As $\veps(x)\neq 0$ for each $x\in U$, we have $\tfrac{C}{2^{N_1}}\veps(x)<\veps(x)$ for each $x\in U$. We deduce
\begin{equation}\label{Loj}
|g^{N_0+N_1}(x)|<|g^{N_0}(x)||g^{N_1}(x)|\leq\frac{1}{2^{N_1}}|g^{N_0}(x)|\leq\frac{C}{2^{N_1}}\veps(x)<\veps(x)
\end{equation}
for each $x\in U$. By Lemma \ref{derlem3}, there exists an integer $N_2>0$ such that
\begin{equation}\label{Loj2}
|D^{(\alpha)}g^{2N_2(N_0+N_1)}(x)|<|g^{N_0+N_1}(x)|<\veps(x)
\end{equation}
for each $x\in U$ and each $\alpha\in\N^n$ such that $|\alpha|\leq\mu$. 

Define $N:=2N_2(N_0+N_1)$. As $|g|<\frac{1}{2}$, $g\neq 0$ on $U$ and $N$ is even, we deduce by \eqref{Loj}
$$
0<g^N(x)<\min\{\veps(x),1\}
$$ 
for each $x\in U$. By \eqref{Loj2} the ${\mathcal S}^{\mu}$ function $h:=g^N:\cl(U)\to\R$ also satisfies $|D^{(\alpha)}h(x)|<\veps(x)$ for each $x\in U$ and each $\alpha\in\N^n$ such that $|\alpha|\leq\mu$, as required.
\end{proof}

We are ready to prove Lemma \ref{zeroNash}.

\begin{proof}[Proof of Lemma \em\ref{zeroNash}]
Let $\veps:U\to\R$ be a strictly positive continuous semialgebraic function bounded by $1$. By Lemma \ref{derlem4} there exists a strictly positive ${\mathcal S}^{\mu}$ function $H:U\to(0,1)$ that is $\tfrac{\veps}{2}$-close to zero with respect to the ${\mathcal S}^{\mu}$ topology. In particular $0<H<\tfrac{\veps}{2}$. By \cite[II.4.1]{sh} there exists a Nash function $h:U\to\R$ which is $H$-close to $H$ with respect to the ${\mathcal S}^{\mu}$ topology, so 
\begin{align*}
&0=H-H<h<H+H=2H<\veps\leq1,\\
&|D^{(\alpha)}h|\leq|D^{(\alpha)}H|+|D^{(\alpha)}(h-H)|<\tfrac{\veps}{2}+H<\veps,
\end{align*}
for each $\alpha\in\N^n$ such that $1\leq|\alpha|\leq\mu$, as required.
\end{proof}

\subsection{Pushing a Nash manifold with corners using Nash diffeomorphisms}
Let $M\subset\R^n$ be a $d$-dimensional Nash manifold and let $\Qq\subset M$ be a $d$-dimensional Nash manifold with corners, which is a closed subset of $M$. Let $\sigma:\Qq\times[0,1]\to\Qq$ be the Nash map introduced in Theorem \ref{push} and let $\Sigma:M_0\times[0,1]\to M$ be the Nash extension of $\sigma$ to $M_0\times[0,1]$, where $M_0\subset M$ is a suitable open semialgebraic neighborhood of $\Qq$. For each semialgebraic function $\delta:M_0\to[0,1]$ we define the semialgebraic maps 
\begin{align}
\sigma_{\delta|_\Qq}:\Qq\to\Qq,\ x\mapsto\sigma(x,\delta(x)),\label{sigmadelta}\\
\Sigma_{\delta}:M_0\to M,\ x\mapsto\Sigma(x,\delta(x)).\label{Sigmadelta}
\end{align}
Observe that $\Sigma_{\delta}$ is a semialgebraic extension of $\sigma_{\delta|_\Qq}$ to $M_0$. Both semialgebraic maps above are Nash if $\delta$ is Nash. If $\delta(x)>0$ for each $x\in\Qq$, we have by Theorem \ref{push} $\sigma_\delta(\Qq)\subset\Int(\Qq)$, whereas if $\delta=0$ we have $\sigma_{\delta|_\Qq}=\id_\Qq$ and $\Sigma_{\delta}$ is the inclusion of $M_0$ in $M$. In addition, we have the following.

\begin{cor}\label{approxderrem}
Let $\veps:M_0\to\R$ be a strictly positive continuous semialgebraic function and let $\mu\geq1$ be an integer. Then there exists a strictly positive continuous semialgebraic function $\eps:M_0\to\R$ such that if $\delta:M_0\to\R$ is an ${\mathcal S}^{\mu}$ function $\eps$-close to zero with respect to the ${\mathcal S}^{\mu}$ topology, then $\Sigma_{\delta}$ is $\veps$-close with respect to the ${\mathcal S}^{\mu}$ topology to the inclusion ${\tt j}_{M_0}:M_0\hookrightarrow M$. 

In addition, if $\delta$ is a Nash function and $\veps$ is small enough, $\Sigma_{\delta}:M_0\to M$ is a Nash embedding.
\end{cor}
\begin{proof}
It is enough to use \S\ref{emb} and the continuity of the map 
$$
\Sigma_*:{\mathcal S}^{\mu}(M_0)\to{\mathcal S}^{\mu}(M_0,M),\ \eta\mapsto\Sigma\circ({\tt j}_{M_0},\eta)
$$
with respect to the ${\mathcal S}^{\mu}$ topology, where ${\tt j}_{M_0}:M_0\hookrightarrow M$ is the inclusion.
\end{proof}

As an application of Theorem \ref{push} and Lemma \ref{zeroNash}, we show how to push a Nash manifold with corners $\Qq$ inside its interior $\Int(\Qq)$ using a one-parameter family of Nash diffeomorphisms (onto their images) arbitrarily close to the identity map with respect to the ${\mathcal S}^{\mu}$ topology. 

\begin{thm}[Pushing Nash manifolds with corners using Nash diffeomorphisms]\label{pushdiffeo}
Let $M\subset\R^n$ be a $d$-dimensional Nash manifold and let $\Qq\subset M$ be a $d$-dimensional Nash manifold with corners that is closed in $M$. Let $\veps:M_0\to\R$ be a strictly positive continuous semialgebraic function and let $\mu\geq1$ be an integer. Then there exists a Nash function $\delta:M_0\to(0,1)$ such that the Nash~map
$$
\Psi:M_0\times [0,1]\to M,\ (x,t)\mapsto \Psi_t(x):=\Sigma_{t\delta}(x)=\Sigma(x,t\delta(x))
$$
satisfies
\begin{itemize}
\item[{\rm (i)}] $\Psi_0={\tt j}_{M_0}$ and $\Psi_t(\Qq)\subset\Int(\Qq)$ for each $t\in(0,1]$,
\item[{\rm (ii)}] $\Psi_t:M_0\hookrightarrow M$ is a Nash embedding for each $t\in[0,1]$,
\item[{\rm (iii)}] $\Psi_t$ is $\veps$-close to the inclusion ${\tt j}_{M_0}:M_0\hookrightarrow M$ with respect to the ${\mathcal S}^{\mu}$ topology for each $t\in [0,1]$.
\end{itemize}
\end{thm}
\begin{proof}
When we say that two functions are `close', we mean that they are close with respect to the corresponding ${\mathcal S}^{\mu}$ topology. For each semialgebraic function $\eta:M_0\to(0,1)$ let $\Sigma_{\eta}$ be the map introduced in \eqref{Sigmadelta}. By Corollary \ref{approxderrem} there exists a strictly positive continuous semialgebraic function $\eps:M_0\to\R$ such that each Nash function $\eta:M_0\to(0,1)$ that is $\eps$-close to zero satisfies that $\Sigma_{\eta}$ is a Nash embedding $\veps$-close to the inclusion map ${\tt j}_{M_0}:M_0\hookrightarrow M$. By Lemma \ref{zeroNash} there exists a Nash function $\delta:M_0\to(0,1)$ that is $\eps$-close to zero. We deduce that $t\delta$ is $\eps$-close to zero as well for each $t\in [0,1]$. Thus, $\Psi_t:=\Sigma_{t\delta}:M_0\to M$ is a Nash embedding and $\veps$-close to the inclusion map ${\tt j}_{M_0}$ for each $t\in[0,1]$. In addition, $\Psi_0={\tt j}_{M_0}$ and by Theorem \ref{push} $\Psi_t(\Qq)=\sigma_{t\delta}(\Qq)\subset\Int(\Qq)$ for each $t\in(0,1]$, because $\im(t\delta)\subset(0,1)$ for each $t\in(0,1]$, as required.
\end{proof}
\begin{remark}
The Nash map
$$
\widehat{\Sigma}:M_0\times[0,1]\to M\times[0,1],\ (x,t)\mapsto(\Sigma(x,t\delta(x)),t)
$$
is a Nash embedding, if the strictly positive Nash function $\delta:M_0\to(0,1)$ is close enough to $0$ with respect to the ${\mathcal S}^\mu$ topology.\hfill$\sqbullet$
\end{remark}

As a straightforward consequence of Corollary \ref{approxderrem} and Theorem \ref{pushdiffeo} we have the following:

\begin{cor}\label{approxder}
Let $\veps:\Qq\to\R$ be a strictly positive continuous semialgebraic function and let $\mu\geq 0$ be an integer. Then there exists a Nash function $\delta:M_0\to(0,1)$ such that the Nash map 
$$
\sigma_{\delta|_\Qq}:\Qq\to\Int(\Qq),\ x\mapsto\sigma(x,\delta(x))
$$ 
is a Nash embedding that is $\veps$-close to the identity map $\id_{\Qq}$ with respect to the ${\mathcal S}^{\mu}$ topology. 
\end{cor}

\section{Proofs of the main results}\label{s6}

In this section we prove Theorems \ref{thm1} and \ref{thm2}. We first present the proof of Theorem \ref{thm1} and to lighten the proof of Theorem \ref{thm2} we present before some preliminary results.

\subsection{Nash approximation when the target space is a Nash manifold with corners}

We use the techniques developed in the previous section to improve \cite[Thm.1.5]{cf3}. Fix an integer $\mu\geq 0$. In the following when we say that two functions are `close', we mean that they are close with respect to the corresponding ${\mathcal S}^{\mu}$ topology. We are ready to prove Theorem \ref{thm1}.

\begin{proof}[Proof of Theorem \em \ref{thm1}]
The proof is conducted in two steps:

\noindent{\sc Step 1.} Let $\veps:\Ss\to\R$ be a strictly positive continuous semialgebraic function. Let 
$$
\sigma:\Qq\times[0,1]\to\Qq,\ (x,t)\mapsto\sigma_t(x)
$$
be the Nash map introduced in Theorem \ref{push}. By Proposition \ref{leftcomp} the map 
$$
\sigma_*:{\mathcal S}^{\mu}(\Ss,\Qq\times [0,1])\to{\mathcal S}^{\mu}(\Ss,\Qq),\ (h,\eta)\mapsto\sigma\circ(h,\eta)
$$ 
is continuous with respect to the ${\mathcal S}^{\mu}$ topology. Thus, there exists a strictly positive continuous semialgebraic function $\eps:\Ss\to\R$ such that if $(h,\eta)\in{\mathcal S}^{\mu}(\Ss,\Qq\times [0,1])$ is an ${\mathcal S}^{\mu}$ map that is $\eps$-close to $(f,0)$, then $\sigma\circ(h,\eta)$ is $\tfrac{\veps}{2}$-close to 
$$
\sigma\circ(f,0)=\id_{\Qq}\circ f=f.
$$
As $\Ss$ is a locally compact semialgebraic subset of $\R^m$, the boundary $\delta^\bullet\Ss:=\cl(\Ss)\setminus\Ss$ is a closed semialgebraic subset of $\R^m$, so $\Ss$ is closed in $U:=\R^m\setminus\delta^\bullet\Ss$. Let $\eps^*:U\to\R$ be a strictly positive continuous semialgebraic extension of $\eps$. By Lemma \ref{zeroNash} there exists a Nash function $\delta:U\to(0,1)$ which is $\eps^*$-close to zero, so $\sigma\circ(f,\delta|_{\Ss})$ is $\tfrac{\veps}{2}$-close to $f$. As $\delta(x)>0$ for each $x\in\Ss$ and by Theorem \ref{push} $\sigma_t(x)\in\Int(\Qq)$ for each $(x,t)\in\Qq\times (0,1]$, we deduce that $\sigma(f(x),\delta(x))\in\Int(\Qq)$ for each $x\in\Ss$. After replacing $f$ by $\sigma\circ(f,\delta|_{\Ss})$ and $\veps$ by $\tfrac{\veps}{2}$, we may assume: $f(\Ss)\subset\Int(\Qq)$.

\noindent{\sc Step 2.} As $f\in{\mathcal S}^{\mu}(\Ss,\Int(\Qq))$, there exists an open semialgebraic neighborhood $V$ of $\Ss$ in $\R^m$ such that $f$ extends to an ${\mathcal S}^{\mu}$ map $F:V\to\R^n$. Let $(\Omega,\nu)$ be a Nash tubular neighborhood of $\Int(\Qq)$ in $\R^n$, see \cite[Cor.8.9.5]{bcr}. As $\Omega$ is an open semialgebraic neighborhood of $\Int(\Qq)$ in $\R^n$, then $F^{-1}(\Omega)$ is an open semialgebraic neighborhood of $\Ss$ in $V$. After replacing $V$ by $V\cap F^{-1}(\Omega)$ if necessary, we may assume that $F(V)\subset\Omega$. As $\Ss$ is closed in $U$, we deduce $\Ss$ is closed in $U\cap V$. After replacing $V$ by $V\cap U$ if necessary, we may assume that $\Ss$ is closed in $V$. Thus, $F:V\to\Omega$ is an ${\mathcal S}^{\mu}$ extension of $f:\Ss\to\Int(\Qq)$. As $\Ss$ is closed in $V$, there exists by \cite[Prop.2.6.9]{bcr} a strictly positive continuous semialgebraic extension $\veps^*:V\to\R$ of $\veps$. 

By Proposition \ref{leftcomp} the map
$$
\nu_*:{\mathcal S}^{\mu}(V,\Omega)\to{\mathcal S}^{\mu}(V,\Int(\Qq)),\ G\mapsto \nu \circ G
$$
is continuous with respect to the ${\mathcal S}^{\mu}$ topology. In particular, there exists a strictly positive continuous semialgebraic function $\eps^*:V\to\R$ such that if $G\in{\mathcal S}^{\mu}(V,\Omega)$ is $\eps^*$-close to $F$, then $\nu\circ G$ is $\veps^*$-close to $\nu\circ F$. As $V$ and $\Int(\Qq)$ are Nash manifolds, there exists by \cite[II.4.1]{sh} a Nash map $G:V\to\Omega$ that is $\eps^*$-close to $F$. Thus, the Nash map $g:=\nu\circ G|_{\Ss}:\Ss\to\Int(\Qq)$ is $\veps$-close to $\nu\circ F|_{\Ss}=f$, as required.
\end{proof}

\subsection{Nash equations of Nash subsets close to zero}

In the proof of Theorem \ref{thm2} we need Nash equations of a Nash subset $X$ of a Nash manifold $N$ that are close to zero with respect to the ${\mathcal S}^{\mu}$ topology. Let $\mu\geq 0$ be an integer. When we say that two functions are `close', we mean that they are close with respect to the corresponding ${\mathcal S}^{\mu}$ topology.

\begin{prop}[Nash equations of Nash subsets close to zero]\label{niceequations}
Let $N\subset\R^m$ be a Nash manifold and $X\subset N$ a Nash subset. Then, for each integer $\mu\geq 0$ there exist Nash functions $\varphi:N\to [0,1)$ arbitrarily close to zero with respect to the ${\mathcal S}^{\mu}$ topology and such that $X=\{\varphi=0\}$. 
\end{prop}
\begin{proof}
Let $\veps:N\to\R$ be a strictly positive continuous semialgebraic function. After replacing $\veps$ by $\min\{\veps,1/2\}$ if necessary, we may assume $\veps<1$. After replacing $N$ by a Nash tubular neighborhood $(\Omega,\nu)$ of $N$ in $\R^m$, the Nash set $X$ by $\nu^{-1}(X)$ and $\veps$ by $\veps\circ \nu$, we may assume: \textit{$N$ is an open semialgebraic subset of $\R^m$.} Let $\psi:N\to\R$ be a Nash equation for $X$, that is, $X=\{\psi=0\}$. After replacing $\psi$ by $\frac{\psi^2}{1+\psi^2}$, we may assume $\psi(N)\subset[0,1)$. Define the non-negative continuous semialgebraic function
$$
\eps:N\to\R,\ x\mapsto\max_{0\leq|\alpha|\leq\mu}\{|D^{(\alpha)}\psi(x)|\}.
$$
Let $f:N\to\R$ be a Nash function. For each $x\in N$ we have by multivariate Leibniz's rule \eqref{Leibniz} that
\begin{align}
\begin{split}\label{stimasmoothing}
|D^{(\alpha)}(\psi f)(x)|&\leq \displaystyle\sum_{\beta\leq\alpha}|D^{(\alpha-\beta)}\psi(x)||D^{(\beta)}f(x)|\\
&\leq \eps(x)\cdot\sum_{\beta\leq\alpha}|D^{(\beta)}f(x)|\leq \max\{\eps(x),1\}\cdot\displaystyle\sum_{\beta\leq \alpha}|D^{(\beta)}f(x)|
\end{split}
\end{align}
for each $\alpha\in\N^m$ such that $|\alpha|\leq\mu$. The number of multi-indices $\beta\in\mathbb{N}^m$ such that $\beta\leq\alpha$ is $m^{|\alpha|}$. Thus, if $|\alpha|\leq\mu$, the number of multi-indices $\beta\in\mathbb{N}^m$ such that $\beta\leq\alpha$ is $\leq m^{\mu}$. Define 
$$
\veps^*:=\frac{\veps}{\max\{m,2\}^{\mu+1} \max\{\eps,1\}},
$$
which satisfies $\veps^*<\veps$. By Lemma \ref{zeroNash} there exists a Nash function $f:N\to(0,1)$ that is $\veps^*$-close to zero. By \eqref{stimasmoothing} we deduce that the Nash function $(f\psi):N\to[0,1)$ is $\veps$-close to zero. As $f$ is strictly positive, we deduce $\{f\psi=0\}=\{\psi=0\}=X$, as required. 
\end{proof}

\subsection{Proof of Theorem \ref{thm2}}
Before proving Theorem \ref{thm2} we need some preliminary results. Recall the definitions concerning distances between semialgebraic sets already presented in the Introduction. The following result holds similarly for arbitrary sets, but we only prove it for semialgebraic sets, because it is the setting considered in this article.

\begin{lem}\label{dist}
Let $C\subset\R^n$ be a closed semialgebraic set such that $C=\cl(\Int(C))$. For each $x\in\Int(C)$ we have $\Bb_n(x,\dist(x,\partial C))\subset\Int(C)$, where $\partial C:=C\setminus\Int(C)$.
\end{lem}
\begin{proof}
Observe that $\{\R^n\setminus C,\partial C,\Int(C)\}$ is a semialgebraic partition of $\R^n$. Consider the bounded semialgebraic function 
$$
f:\R^n\to\{-1,0,1\},\ x\mapsto
\begin{cases}
-1&\text{if $x\in\R^n\setminus C$},\\
0&\text{if $x\in\partial C$},\\
1&\text{if $x\in\Int(C)$},
\end{cases}
$$
which is continuous on $\R^n\setminus\partial C$. The continuous semialgebraic function $\dist(\cdot,\partial C)$ has zero set equal to $\partial C$, so $g:=\dist(\cdot,\partial C)\cdot f$ is a continuous semialgebraic function on $\R^n$ and $\{g<0\}=\R^n\setminus C$, $\{g=0\}=\partial C$ and $\{g>0\}=\Int(C)$. Pick $x\in\Int(C)$, write $0<r:=\dist(x,\partial C)$ and suppose $\Bb_n(x,r)\not\subset\Int(C)$. Pick $y\in\Bb_n(x,r)\cap(\R^n\setminus C)$ and let $L$ be the line that passes through $x$ and $y$. Consider the parametrization $h:\R\to L,\ t\mapsto x+t(y-x)$ and the composition $g\circ h:\R\to\R$, which takes the values $g(x)>0$ at $t=0$ and $g(y)<0$ at $t=1$. Thus, there exists $t_0\in(0,1)$ such that $(g\circ h)(t_0)=0$, so $z:=h(t_0)\in\partial C$ belongs to the interior of the segment that connects $x$ and $y$. Consequently, $z\in\partial C\cap\Bb_n(x,r)$, so $\dist(x,\partial C)<r$, which is a contradiction. We deduce $\Bb_n(x,r)\subset\Int(C)$, as required. 
\end{proof}

\begin{lem}\label{qast}
Let $\Qq\subset\R^n$ be a $d$-dimensional Nash manifold with corners and let $M$ be a Nash envelope of $\Qq$ such that $\Qq$ is a closed subset of $M$, $M$ is a closed subset of $\R^n$ and the Nash closure $X$ of $\partial\Qq$ is a Nash normal-crossings divisor of $M$ such that $X\cap\Qq=\partial\Qq$. Let $(\Ee,\eta)$ be a Nash tubular neighborhood of $M$ in $\R^n$ (as the one described in \cite[Prop.8.9.2 \& 8.9.3]{bcr}) and let $\tau:M\to\R$ be a strictly positive Nash function such that $T:=\{x\in\Ee:\ \|x-\eta(x)\|\leq(\tau\circ\eta)(x)\}$ is a closed subset of $\R^n$. Then $\Qq^*:=\{x\in\R^n:\ \eta(x)\in\Qq,\ \|x-\eta(x)\|^2\leq(\tau\circ\eta)^2(x)\}$ is a Nash manifold with corners and it is a closed subset of $\R^n$. 
\end{lem}
\begin{proof}
Let $X_1,\ldots,X_s$ be the irreducible components of $X$. Let $h_i:M\to\R$ be a Nash equation of $X_i$ such that $d_ph_i:T_pM\to\R$ is surjective for each $p\in X_i$ and each $i=1,\ldots,s$ and $\Qq=\{h_1\geq0,\ldots,h_s\geq0\}$ (see \cite[Cor.2.12]{cf3}). Define $h_i^*:=h_i\circ\eta:\Ee\to\R$ for $i=1,\ldots,s$. Let $W:=\{(y,v)\in M\times\R^n:\ v\in(T_yM)^\bot\}$, let $\Pi:W\to M,\ (y,v)\mapsto y$ and let $\varphi:W\to\R^n,\ (y,v)\mapsto y+v$. By \cite[Prop.8.9.3]{bcr} there exists an open semialgebraic neighborhood $U\subset W$ such that $\varphi|_U:U\to\Ee$ is a Nash diffeomorphism and for each $(y,v)\in U$ and each $t\in[0,1]$ we have $(y,tv)\in U$. In addition, by \cite[Cor.8.9.5]{bcr} it holds $\eta=\Pi\circ(\varphi|_U)^{-1}$. Define $\Qq^\bullet:=\eta^{-1}(\Qq)=(\varphi|_U)\circ(\Pi^{-1}(\Qq)\cap U)$. As $\varphi|_U$ is a Nash diffeomorphism, $\Qq^\bullet$ is a Nash manifold with corners if and only if $\Pi^{-1}(\Qq)\cap U$ is a Nash manifold with corners. To that end it is enough to prove: {\em $\Pi^{-1}(\Qq)=\{h_1\circ\Pi\geq0,\ldots,h_s\circ\Pi\geq0\}$ is a Nash manifold with corners}. 

By \cite[Prop.8.9.2]{bcr} there exists an open semialgebraic covering $\{U_1,\ldots,U_m\}$ of $M$ and Nash diffeomorphisms $\theta_i:U_i\times\R^{n-d}\to\Pi^{-1}(U_i)$ for $i=1,\ldots,m$ such that $\Pi\circ\theta_i:U_i\times\R^{n-d}\to U_i$ is the projection onto the first factor and $\theta_i|_{\{y\}\times\R^{n-d}}:\{y\}\times\R^{n-d}\to\{y\}\times(T_yM)^\bot$ is an $\R$-isomorphism for each $y\in U_i$. Using Gram-Schmidt's algorithm we may assume $\theta_i|_{\{y\}\times\R^{n-d}}$ is an $\R$-isometry for each $y\in U_i$. It is enough to check: {\em $\theta_i^{-1}(\Pi^{-1}(\Qq\cap U_i))=\{h_1\circ\Pi\circ\theta_i\geq0,\ldots,h_s\circ\Pi\circ\theta_i\geq0\}=(\Qq\cap U_i)\times\R^{n-d}$ is a Nash manifold with corners for each $i=1,\ldots,m$}, which is trivially true, because each intersection $(\Qq\cap U_i)$ is a Nash manifold with corners.

Define $h_{s+1}:\Ee\to\R,\ x\mapsto(\tau\circ\eta)^2(x)-\|x-\eta(x)\|^2$, which is a Nash function on $\Ee$, and $h_{s+1}^*:W\to\R,\ (y,v)\mapsto\tau^2(y)-\|v\|^2$, which is a Nash function on $W$ that satisfies $h_{s+1}=h_{s+1}^*\circ(\varphi|_U)^{-1}$, because $\eta\circ(\varphi|_U)=\Pi$ and $\varphi(y,v)-\eta\circ\Pi(y,v)=v$ for each $(y,v)\in W$. Thus, for each $(x,u)\in U_i\times\R^{n-d}$, we deduce 
$$
(h_{s+1}^*\circ\theta_i)(y,u)=\tau^2(y)-\|u\|^2,
$$
because $\theta_i|_{\{y\}\times\R^{n-d}}$ is an $\R$-isometry for each $y\in U_i$ and each $i=1,\ldots,m$. Let us check: {\em 
\begin{multline*}
((\Qq\cap U_i)\times\R^{n-d})\cap\{h_{s+1}^*\circ\theta_i\geq0\}\\
=\{(y,u)\in U_i\times\R^{n-d}:\ h_1(y)\geq0,\ldots,h_s(y)\geq0,\tau^2(y)-\|u\|^2\geq0\}
\end{multline*} 
is a Nash manifold with corners}. 

It is enough to show that $(X\times\R^{n-d})\cup\{(y,u)\in U_i\times\R^{n-d}:\ \tau^2(y)-\|u\|^2=0\}$ is a Nash normal-crossings divisor. This holds, because $\frac{\partial h_j}{\partial u_k}(y)=0$ and $\frac{\partial(h_{s+1}^*\circ\theta_i)}{\partial u_k}(y,u)=-2u_k$ for each $k=1,\ldots,n-d$ and each $j=1,\ldots,s$ and $\|u\|^2=\tau^2(y)>0$ for each $(y,u)\in U_i\times\R^{n-d}$ such that $(h_{s+1}^*\circ\theta_i)(y,u)=0$.

Putting all the previous information together we deduce that $\Qq^*$ is a Nash manifold with corners. To finish observe that as $\Qq$ is a closed subset of $M$, the inverse image $\eta^{-1}(\Qq)$ is a closed subset of $\Ee$, so $\eta^{-1}(\Qq)\cap T$ is a closed subset of $T$. As $T$ is a closed subset of $\R^n$ and $\Qq^*=\eta^{-1}(\Qq)\cap T$, we conclude that $\Qq^*$ is a closed subset of $\R^n$, as required.
\end{proof}

We are ready to prove Theorem \ref{thm2}. When we say that two functions are `close' we mean that they are close with respect to the corresponding ${\mathcal S}^{\mu}$ topology.

\begin{proof}[Proof of Theorem \em \ref{thm2}]
The proof is conducted in several steps:

\noindent{\sc Step 1. Initial reduction.}
Let us check that we may assume: {\em $\Qq\subset\R^n$ is a closed subset of $\R^n$ with non-empty interior in $\R^n$.}

Let $M$ be a Nash envelope of $\Qq$ in $\R^n$ such that $\Qq$ is closed in $M$ and the Nash closure $Y$ of $\partial\Qq$ in $M$ is a Nash normal-crossings divisor such that $\Qq\cap Y=\partial\Qq$ (Theorem \ref{divisorialPre}). By Mostowski's trick (Proposition \ref{Mos}) we may assume that $M$ is a closed subset of $\R^n$. Let $(\Ee,\eta)$ be a Nash tubular neighborhood of $M$ in $\R^n$ provided by \cite[Prop.8.9.2 \& 8.9.3]{bcr}. We claim: {\em After shrinking $N$ if necessary, there exists an $\mathcal{S}^{\mu}$ extension $F:N\to\Ee$ of $f:\Ss\to\Qq$ to $N$ that is Nash on an open semialgebraic neighborhood $U_0\subset U$ of $X\cap\Ss$.}

By Proposition \ref{Mos} we may assume that $N$ is a closed Nash submanifold of $\R^m$. As $f$ is Nash on $\Ss\cap U$, there exists an open semialgebraic neighborhood $U_1\subset U$ of $\Ss\cap U$ and a Nash extension $F_1:U_1\to \R^n$ of $f|_{\Ss\cap U}:\Ss\cap U\to\Qq$ to $U_1$. Analogously, as $f$ is $\mathcal{S}^{\mu}$ on $\Ss$, there exists an open semialgebraic neighborhood $U_2\subset N$ of $\Ss$ and an $\mathcal{S}^{\mu}$ extension $F_2:U_2\to\R^n$ of $f:\Ss\to\Qq$ to $U_2$. After replacing $N$ by $U_2$, we may assume that $F_2:N\to \R^n$. As $N$ is closed in $\R^n$ and both $\Ss$ and $X$ are closed in $N$, also $\Ss\cap X$ is closed in $\R^n$. Let $U_0\subset U_1$ be an open semialgebraic neighborhood of $\Ss\cap X$ in $U_1$ such that $\cl(U_0)\subset U_1$. Let $\{\theta_1,\theta_2\}$ be an $\mathcal{S}^{\mu}$ partition of unity subordinated to the covering $\{U_1,N\setminus\cl(U_0)\}$. Define $F:=\theta_1 F_1+\theta_2 F_2:N\to \R^n$, which is an $\mathcal{S}^{\mu}$ extension of $f$ to $N$. After shrinking $N$ if necessary, we may assume $F(N)\subset\Ee$. As $\cl(U_0)\subset U_1$, we have $F=F_1$ on $U_0$. We conclude that $F$ is Nash on $U_0$, as claimed.

Consider the strictly positive continuous semialgebraic function $\dist(\cdot,\R^n\setminus\Ee):M\to(0,+\infty)$. By \cite[\S1]{ef} there exists a Nash function $\tau:M\to\R$ such that $0<\frac{1}{4}\dist(\cdot,\R^n\setminus\Ee)<\tau<\frac{3}{4}\dist(\cdot,\R^n\setminus\Ee)$, so in particular $\tau$ is also strictly positive. Recall that by \cite[Cor.8.9.5]{bcr} $\dist(x,M)=\|x-\eta(x)\|$ for each $x\in\Ee$. Let us check: {\em $T:=\{x\in\Ee:\ \|x-\eta(x)\|\leq\tau(x)\}$ is a closed subset of $\R^n$.} Once this is proved, Lemma \ref{qast} implies that $\Qq^*:=\{x\in\R^n:\ \eta(x)\in\Qq,\ \|x-\eta(x)\|^2\leq(\tau\circ\eta)^2(x)\}$ is a Nash manifold with corners and it is a closed subset of $\R^n$.

Pick a point $y\in\cl(T)$. By the curve selection lemma \cite[Thm.2.5.5]{bcr} there exists a continuous semialgebraic arc $\alpha:(-1,1)\to\R^n$ such that $\alpha(0)=y$ and $\alpha((0,1))\subset T$. We have $\dist(\alpha(t),M)=\|\alpha(t)-\eta(\alpha(t)))\|\leq\tau(\alpha(t))<\frac{3}{4}\dist(\alpha(t),\R^n\setminus\Ee)$ for each $t\in(0,1)$. Thus, $\dist(y,M)=\dist(\alpha(0),M)\leq\frac{3}{4}\dist(\alpha(0),\R^n\setminus\Ee)=\frac{3}{4}\dist(y,\R^n\setminus\Ee)$. We deduce $\dist(y,\R^n\setminus\Ee)>0$, because otherwise $\dist(y,M)=0$ and $\dist(y,\R^n\setminus\Ee)=0$, so $y\in M\cap(\R^n\setminus\Ee)=\varnothing$, which is a contradiction. Thus, $y\in\Ee$, so $y\in\cl(T)\cap\Ee=T$, because $T$ is a closed subset of $\Ee$. We conclude $T$ is a closed subset of $\R^n$.
 
Recall that $\Ss$ is a closed semialgebraic subset of $N$, so each strictly positive continuous semialgebraic function $\Ss\to \R$ extends to a strictly positive continuous semialgebraic function $N\to \R$. In particular, we may fix a strictly positive continuous semialgebraic function $\veps:N\to\R$. By Proposition \ref{leftcomp} the map $\eta_*:{\mathcal S}^{\mu}(N,\Ee)\to{\mathcal S}^{\mu}(N,M),\ H\mapsto\eta\circ H$ is continuous. Thus, there exists a strictly positive continuous semialgebraic function $\delta:N\to\R$ such that if $H\in{\mathcal S}^{\mu}(N,\Ee)$ is $\delta$-close to $F$ and $H(\Ss)\subset\eta^{-1}(\Qq)$, then $\eta\circ H$ is $\veps$-close to $\eta\circ F$ and $\eta\circ H|_\Ss:\Ss\to\Qq$ is $\veps|_\Ss$-close to $\eta\circ f=f:\Ss\to\Qq$. 

We replace $M$ by $\Ee$, $X$ by $\eta^{-1}(X)$, $\Qq$ by $\Qq^*$ and $\veps$ by $\delta$, so we may assume: {\em $M$ is an open subset of $\R^n$ and $\Qq$ is a closed subset of $\R^n$ with non-empty interior in $\R^n$ such that $\cl(\Int(\Qq))=\Qq$} (in particular, Lemma \ref{dist} applies to $\Qq$).

\noindent{\sc Step 2. Reduction to the case of a nice intersection with the boundary.} Let $\varphi:N\to [0,1)$ be a Nash equation in $N$ of the Nash set $X$. Let 
\begin{align*}
&\sigma:\Qq\times [0,1]\to\Qq,\ (x,t)\mapsto\sigma_t(x):=\sigma(x,t),\\
&\Sigma:M_0\times [0,1]\to M,\ (x,t)\mapsto\Sigma_t(x):=\Sigma(x,t)
\end{align*}
be the Nash maps introduced in Theorem \ref{push}, where $M_0\subset M$ is an open semialgebraic neighborhood of $\Qq$ and $\sigma=\Sigma|_{\Qq\times[0,1]}$. Recall that $\sigma_0=\id_{\Qq}$ and $\sigma_t(\Qq)\subset\Int(\Qq)$ for each $t\in(0,1]$. After shrinking $N$ if necessary, we may assume $F(N)\subset M_0$.

As $\varphi(x)\in [0,1)$ for each $x\in N$, the composition 
$$
\Sigma_\varphi\circ F:N\to M,\ x\mapsto\Sigma_{\varphi(x)}(F(x)):=\Sigma(F(x),\varphi(x))
$$
is a well-defined ${\mathcal S}^{\mu}$ map that is Nash on $U_0$ and $\sigma_{\varphi|_\Ss}\circ f=\Sigma_{\varphi|_\Ss}\circ F|_\Ss:\Ss\to\Qq$. 

By Proposition \ref{leftcomp} the map 
$$
\sigma_*:{\mathcal S}^{\mu}(N,M_0\times [0,1])\to{\mathcal S}^{\mu}(N,M),\ (H,\rho)\mapsto\sigma\circ(H,\rho)
$$ 
is continuous. Thus, there exists a strictly positive continuous semialgebraic function $\eps_0:N\to\R$ strictly smaller than $\frac{1}{2}$ such that if $(H,\rho)\in{\mathcal S}^{\mu}(N,M_0\times [0,1])$ is an ${\mathcal S}^{\mu}$ map that is $\eps_0$-close to $(F,0)$ and $H(\Ss)\subset\Qq\times[0,1]$, then $\Sigma_{\rho}\circ H=\Sigma\circ(H,\rho)$ is $\tfrac{\veps}{2}$-close to 
$$
\Sigma_0\circ F=\Sigma\circ(F,0)={\tt j}_{M_0}\circ F=F,
$$
where ${\tt j}_{M_0}:M_0\hookrightarrow M$ is the inclusion, and $\sigma_{\rho|_\Ss}\circ H|_\Ss:\Ss\to\Qq$ is $\tfrac{\veps}{2}$-close to $\sigma_0\circ f=\sigma\circ(f,0)=f$. By Proposition \ref{niceequations} there exists a Nash equation $\varphi_0:N\to [0,1)$ for the Nash set $X$ that is $\eps_0$-close to zero. In particular, $\Sigma_\varphi\circ F$ is $\frac{\veps}{2}$-close to $F$ and if $x\in X$, then $\Sigma_\varphi\circ F(x)=\Sigma(F(x),0)=F(x)$. We claim: $(\Sigma_\varphi\circ F)^{-1}(\partial\Qq)\cap\Ss\subset\varphi^{-1}(0)=X$.

If $x\in\Ss$ satisfies $\varphi(x)\neq0$, then $F(x)=f(x)\in\Qq$ and $\Sigma(F(x),\varphi(x))=\sigma(f(x),\varphi(x))\in\Int(\Qq)$ (because $\sigma_t(\Qq)\subset\Int(\Qq)$ for each $t\in(0,1]$). Thus, $(\Sigma_\varphi\circ F)^{-1}(\partial\Qq)\cap\Ss\subset\varphi^{-1}(0)=X$, as claimed.

After replacing $F$ by $\Sigma_\varphi\circ F$ and $\veps$ with $\tfrac{\veps}{2}$, we may assume: \textit{$f^{-1}(\partial\Qq)=F^{-1}(\partial\Qq)\cap\Ss\subset X\cap\Ss$.}

\noindent{\sc Step 3. Construction of the approximating function.} 
After the substitution performed in {\sc Step 2}, we have $f^{-1}(\partial\Qq)\subset X\cap\Ss$. Thus,
$$
\{x\in\Ss:\ \dist (f(x),\partial\Qq)=0\}=f^{-1}(\partial\Qq)\subset X\cap\Ss=\{x\in\Ss:\ \varphi_0(x)=0\}.
$$
By \cite[Thm.2.6.6]{bcr} and its proof there exist an integer $p_0\geq 0$ and a continuous semialgebraic function $h:\Ss\to\R$ such that 
\begin{itemize}
\item $X\cap\Ss=\{h=0\}$,
\item $\varphi_0^{p_0}(x)=h(x)\dist(f(x),\partial\Qq)$ for each $x\in\Ss$.
\end{itemize}
In particular, $W:=\{x\in\Ss:\ |h(x)|<1\}$ is an open semialgebraic neighborhood of $X\cap\Ss$ in $\Ss$. As $f^{-1}(\partial\Qq)\subset X\cap\Ss$ and $W$ is an open semialgebraic neighborhood of $X\cap\Ss$ in $\Ss$, if $\Ss\setminus W\neq \varnothing$, the continuous semialgebraic function 
$$
\eps_1^*:\Ss\setminus W\to\R,\ x\mapsto \dist(f(x),\partial\Qq)
$$
is strictly positive. Let $\eps_1:N\to\R$ be a strictly positive continuous semialgebraic extension of $\eps_1^*$ (defined on the closed semialgebraic subset $\Ss\setminus W$ of $N$) to $N$. If $N\setminus W=\varnothing$ we set $\eps_1\equiv 1$. By Proposition \ref{niceequations} and using generalized Leibniz's formula \eqref{Leibniz} there exists a Nash equation $\varphi_1:N\to [0,1)$ for the Nash set $X$ such that $|\varphi_1|<\frac{1}{2}$ and $\varphi:=\varphi_1\varphi_0^{p_0}$ is $\min\{\eps_0,\eps_1,\frac{1}{2}\}$-close to zero. By \cite[II.5.2]{sh} and its proof (recall that by {\sc Step 1} $M$ is open in $\R^n$) there exist
\begin{itemize}
\item Nash maps $G_0:N\to M$, $\Psi^*:N\to M$ and
\item an ${\mathcal S}^{\mu}$ map $\Psi:N\to M$,
\end{itemize}
such that 
\begin{itemize} 
\item[{\rm (i)}] $\Psi^*$ is $\min\{\veps,\frac{1}{2}\}$-close to $\Psi$,
\item[{\rm (ii)}] the Nash map $G:=G_0-\varphi\Psi^*=F+\varphi(\Psi-\Psi^*)$ satisfies $G(N)\subset M$.
\end{itemize}
In addition, $G(x)=F(x)$ for each $x\in X$, because $X=\{\varphi=0\}$. 

We claim: $G(\Ss)\subset\Qq$. We distinguish two cases:

\noindent{\sc Case 1.} Pick a point $x\in W\subset\Ss$ and recall that $M$ is an open subset of $\R^n$ (see {\sc Step 1}). We have
\begin{align*}
\|G(x)-f(x)\|&=|\varphi_1(x)\varphi_0^{p_0}(x)|\|\Psi(x)-\Psi^*(x)\|=|\varphi_1(x)||h(x)|\|\Psi(x)-\Psi^*(x)\|\dist(f(x),\partial\Qq)\\
&\leq\tfrac{1}{2}|h(x)|\dist(f(x),\partial\Qq)\leq\tfrac{1}{2}\dist(f(x),\partial\Qq),
\end{align*}
because $x\in W$. If $f(x)\in\partial\Qq$, then $G(x)=f(x)\in\partial\Qq\subset\Qq$. If $f(x)\not\in\partial\Qq$, then $f(x)\in\Int(\Qq)$. Thus, $G(x)\in\Bb_n(f(x),\dist(f(x),\partial\Qq))$. By Lemma \ref{dist} $G(x)\in\Int(\Qq)$.

\noindent{\sc Case 2.} Pick a point $x\in\Ss\setminus W$ and observe that
\begin{multline*}
\|G(x)-f(x)\|=\|G(x)-F(x)\|\\
=|\varphi_1(x)\varphi_0^{p_0}(x)|\|\Psi(x)-\Psi^*(x)\|<\tfrac{1}{2}\eps_1^*(x)=\tfrac{1}{2}\dist(f(x),\partial\Qq),
\end{multline*}
so $G(x)\in\Bb_n(f(x),\dist(f(x),\partial\Qq))$. Again by Lemma \ref{dist} $G(x)\in\Int(\Qq)$, as claimed.

In addition, $\|G-F\|=|\varphi|\|\Psi-\Psi^*\|\leq\frac{1}{2}\veps$ and $G(x)=F(x)$ for each $x\in X=\{\varphi=0\}$, as required.
\end{proof}

\section{Applications of the main results and further developments}\label{s7}

\subsection{Nash homotopies when the target space is a Nash manifold with corners}

We use Theorem \ref{thm2} to analyze Nash homotopies between semialgebraic maps with values in a Nash manifold with corners. Let $\Ss\subset\R^m$ be a locally compact semialgebraic set and $\Qq\subset\R^n$ a Nash manifold with corners. We next prove Theorem \ref{homotopy2} and Corollaries \ref{homocor1} and \ref{homocor2} already stated in the Introduction.

\begin{proof}[Proof of Theorem \em\ref{homotopy2}]
Let $\veps:\Ss\times[0,1]\to\R$ be a strictly positive continuous semialgebraic function. By Lemma \ref{st} we may assume $\veps$ does not depend on the variable $\t$, that is, $\veps:\Ss\to\R$. Recall that $\Phi:\Ss\times[0,1]\to\Qq$ is a ${\mathcal S}^\mu_{\tt t}$ homotopy between $f$ and $g$. By Proposition \ref{leftcomp} and \S\ref{ctwst}
$$
\Phi_*:{\mathcal S}^\mu_\t(\Ss\times[0,1],\Ss\times[0,1])\to{\mathcal S}^\mu_\t(\Ss\times[0,1],\Qq),\ H\mapsto \Phi\circ H
$$
is a continuous map with respect to the ${\mathcal S}^\mu_\t$ topology. Thus, there exists a strictly positive continuous semialgebraic function $\delta^*:\Ss\times[0,1]\to\R$ such that $\delta^*<\frac{1}{4}$ and if the continuous semialgebraic map $H:\Ss\times[0,1]\to\Ss\times[0,1]$ satisfies $\|(x,t)-H(x,t)\|<\delta^*(x,t)$ for each $(x,t)\in\Ss\times[0,1]$, then
$$
\|\Phi(x,t)-\Phi(H(x,t))\|<\tfrac{1}{2}\veps(x)
$$
for each $(x,t)\in\Ss\times[0,1]$. Again by Lemma \ref{st} we may assume $\delta^*$ does not depend on the variable $\t$, that is, $\delta^*:\Ss\to\R$.

Let $0<\delta_0<\tfrac{1}{4}$ and consider the (strictly increasing) semialgebraic homeomorphism
$$
\eta^*_{\delta_0}:[\delta_0,1-\delta_0]\to [0,1],\ t\mapsto\frac{t-\delta_0}{1-2\delta_0}.
$$
Define the continuous semialgebraic function 
$$
\eta_{\delta_0}:[0,1]\to [0,1],\ t\mapsto \begin{cases}
0, &\text{ if } 0\leq t\leq \delta_0,\\
\eta^*_{\delta_0}(t), &\text{ if } \delta_0\leq t\leq 1-\delta_0,\\
1, &\text{ if } 1-\delta_0\leq t\leq 1.
\end{cases}
$$
The following holds:
\begin{equation}\label{stimamax}
\max_{t\in[\delta_0,1-\delta_0]}\{|t-\eta_{\delta_0}(t)|\}=\frac{\delta_0}{1-2\delta_0}\max_{t\in [\delta_0,1-\delta_0]}\{|1-2t|\}\leq\delta_0\quad\leadsto\quad\max_{t\in[0,1]}\{|t-\eta_{\delta_0}(t)|\}\leq\delta_0.
\end{equation}

By Lemma \ref{derlem4} there exists an ${\mathcal S}^\mu$ function $\delta:\Ss\to\R$ that is arbitrarily close to zero with respect to the ${\mathcal S}^\mu$ topology. Let $\Delta:U\to\R$ be a strictly positive ${\mathcal S}^\mu$ extension of $\delta:\Ss\to\R$ to an open semialgebraic neighborhood $U\subset\R^n$ of $\Ss$ that is close to zero with respect to the ${\mathcal S}^\mu$ topology. By \eqref{stimamax} we have
$$
\|(x,t)-(x,\eta_{\delta(x)}(t))\|=|t-\eta_{\delta(x)}(t)|\leq \max_{t\in[0,1]}\{|t-\eta_{\delta(x)}(t)|\}\leq\delta(x)
$$
for each $(x,t)\in\Ss\times [0,1]$. By multivariate Leibniz's rule \eqref{Leibniz} we have
\begin{equation}
D^{(\alpha)}\Big(\frac{\Delta}{1-2\Delta}(1-2t)\Big)=(1-2t)\sum_{\beta\leq \alpha}\frac{\alpha!}{\beta!(\alpha-\beta)!}(D^{(\beta)}\Delta) (D^{(\alpha-\beta)}((1-2\Delta)^{-1})).
\end{equation}
In addition, by multivariate Fa\`a di Bruno's formula \cite[Thm.2.1]{cs} we have
$$
D^{(\alpha-\beta)}((1-2\Delta)^{-1})=\sum_{k=1}^{|\alpha-\beta|}\frac{(-1)^kk!}{(1-2\Delta)^{k+1}}\sum_{s=1}^{|\alpha-\beta|}\sum_{p_s((\alpha-\beta),k)}(\alpha-\beta)!\prod_{j=1}^s\frac{((-2)D^{(\kappa_j)}(\Delta))^{\ell_j}}{(\ell_j!)\cdot(\kappa_j!)^{\ell_j}}
$$
where 
\begin{multline*}
p_s((\alpha-\beta),k):=\Big\{(\ell_1,\ldots,\ell_s;\kappa_1,\ldots,\kappa_s):\\ 
\ell_j>0,\ 0<\kappa_1<\cdots<\kappa_s,\ \sum_{j=1}^s\ell_j=k,\ \sum_{j=1}^s\ell_j\kappa_j=\alpha-\beta\Big\}.
\end{multline*}
Thus, if $\Delta$ is close enough to zero with respect to the ${\mathcal S}^\mu$ topology, we may assume $\frac{\delta}{1-2\delta}$ is $\delta^*$-close to zero with respect to the ${\mathcal S}^\mu$ topology and $\id_\Ss\times\eta_\delta:\Ss\times[0,1]\to\Ss\times[0,1]$ is $\delta^*$-close to $\id_{\Ss\times[0,1]}$ in the ${\mathcal S}^\mu_\t$ topology. Consequently, $\Phi$ is $\frac{\veps}{2}$-close to $\Phi^*:\Ss\times[0,1]\to\Qq,\ (x,t)\mapsto\Phi(x,\eta_{\delta(x)}(t))$.

Consider the Nash subset $\R^m\times\{0,1\}$ of $\R^m\times\R$ and observe that $\Phi^*|_{\{(x,t)\in\Ss\times[0,1]:\ t<\delta(x)\}}(x,t)=f(x)$ and $\Phi^*|_{\{(x,t)\in\Ss\times[0,1]:\ t>1-\delta(x)\}}(x,t)=g(x)$, so $\Phi^*$ is Nash on $\{(x,t)\in\Ss\times[0,1]:\ t<\delta(x) \text{ or } t>1-\delta(x)\}$, which is an open semialgebraic neighborhood of the Nash set $(\R^m\times\{0,1\})\cap(\Ss\times[0,1])$ in $\Ss\times[0,1]$. By Theorem \ref{thm2} and Remark \ref{rtwst} there exists a Nash map $\Psi:\Ss\times[0,1]\to\Qq$ $\frac{\veps}{2}$-close to $\Phi^*:\Ss\times[0,1]\to\Qq$ with respect to the ${\mathcal S}^\mu_\t$ topology such that $\Psi(x,0)=\Phi(x,0)=f(x)$ and $\Psi(x,1)=\Phi(x,1)=g(x)$ for each $x\in\Ss$.

As $\Phi(x,t)-\Psi(x,t)=(\Phi-\Phi^*)+(\Phi^*-\Psi)$, we conclude that $\Psi$ is a Nash homotopy between $f$ and $g$ which is $\veps$-close to $\Phi$ with respect to the ${\mathcal S}^\mu_\t$ topology, as required.
\end{proof}
\begin{remark}
If the involved semialgebraic homotopy $\Phi:\Ss\times[0,1]\to\Qq$ between $f$ and $g$ is ${\mathcal S}^{\mu}$ for some $\mu\geq 1$, it is natural to wonder if it is possible to find an approximating Nash homotopy $\Psi$ close to $\Phi$ with respect to the ${\mathcal S}^{\mu}$ topology. In the proof of Theorem \ref{homotopy2} we have used Theorem \ref{thm2}. It seems difficult to apply our strategy to approximate homotopies with respect to the ${\mathcal S}^{\mu}$ topology when $\mu\geq 1$, because we need that the homotopy $\Phi$ is Nash on an open semialgebraic neighborhood $U$ of $\Ss\times\{0,1\}$ in $N\times\R$. To that end, we have to replace $\Phi(x,t)$ by $\Phi(x,\eta(x,t))$, where $\eta:N\times[0,1]\to [0,1]$ is an ${\mathcal S}^{\mu}$ function such that $\tfrac{\partial\eta}{\partial\t}$ is zero when $0\leq t<\delta(x)$ or $1-\delta(x)<t\leq1$ for some strictly positive continuous Nash function $\delta:N\to\R$, so in general $\Phi(x,\eta(x,t))$ is not close to $\Phi(x,t)$ with respect to the ${\mathcal S}^{\mu}$ topology for $\mu\geq 1$. This is substantially why we have introduced the ${\mathcal S}^\mu_\t$ topology in \S\ref{ctwst}.\hfill$\sqbullet$
\end{remark}

We deduce from Theorem \ref{homotopy2} the proof of Corollary \ref{homocor1}.

\begin{proof}[Proof of Corollary \em\ref{homocor1}]
As $\Qq$ is a locally compact semialgebraic subset of $\R^n$, we may assume by Mostowski's trick (see Proposition \ref{Mos}) that $\Qq$ is a closed semialgebraic subset of $\R^n$. Under this assumption, there exist by \cite{dk} an open semialgebraic neighborhood $U$ of $\Qq$ in $\R^n$ and a (continuous) semialgebraic retraction $\rho:\cl(U)\to\Qq$. As $\R^n\setminus U$ is closed and $f(\Ss)\subset U$, the continuous semialgebraic function
$$
\veps:\Ss\to\R,\ x\mapsto\tfrac{1}{2}\dist(f(x),\R^n\setminus U)
$$
is by \cite[Prop.2.2.8]{bcr} strictly positive. Consider the continuous semialgebraic homotopy
$$
\Phi:\Ss\times[0,1]\to\R^n,\ (x,t)\mapsto \Phi_t(x):=(1-t)f(x)+tg(x).
$$
If $\|f(x)-g(x)\|<\veps(x)$ for each $x\in\Ss$, then 
\begin{align}\label{homodist}
\begin{split}
\|f(x)-\Phi_t(x)\|&=\|f(x)-((1-t)f(x)+tg(x))\|\\
&=t\|f(x)-g(x)\|\leq\|f(x)-g(x)\|<\veps(x)=\tfrac{1}{2}\dist(f(x),\R^n\setminus U),
\end{split}
\end{align}
so $\Phi_t(x)\in U$ for each $(x,t)\in\Ss\times[0,1]$. Thus, $\Phi(\Ss\times[0,1])\subset U$ and $\rho\circ\Phi:\Ss\times [0,1]\to\Qq$ is a well-defined continuous semialgebraic homotopy between $f=\rho\circ\Phi(\cdot,0)$ and $g=\rho\circ\Phi(\cdot,1)$. By Theorem \ref{homotopy2} there exists a Nash homotopy $\Psi:\Ss\times[0,1]\to\Qq$ between $f$ and $g$, as required. 
\end{proof}

As a consequence of Theorem \ref{thm1}, we prove Corollary \ref{homocor2} after a preliminary lemma.

\begin{lem}\label{detallito}
Let $\Ss\subset\R^n$ be a semialgebraic set and let $\Psi_1:\Ss\times[0,\frac{1}{2}]\to\Tt$ and $\Psi_2:\Ss\times[\frac{1}{2},1]\to\Tt$ be Nash functions such that $\Psi_1(\cdot,\frac{1}{2})=\Psi_2(\cdot,\frac{1}{2})$. Let $m>\mu\geq1$ be integers such that $m$ is odd and define the semialgebraic homeomorphism $\eta_m:[0,1]\to [0,1],\ t\mapsto\frac{(2t-1)^m}{2}+\frac{1}{2}$, whose restriction $\eta_m|_{[0,1]\setminus\{\frac{1}{2}\}}:[0,1]\setminus\{\frac{1}{2}\}\to[0,1]\setminus\{\frac{1}{2}\}$ is a Nash diffeomorphism. Then $\Psi^*:\Ss\times[0,1]\to\Tt,\ (x,t)\mapsto\Psi(x,\eta_m(t))$, where
$$
\Psi:\Ss\times[0,1]\to\Tt, (x,t)\mapsto\begin{cases}
\Psi_1(x,t)&\text{if $t\in[0,\frac{1}{2}]$,}\\
\Psi_2(x,t)&\text{if $t\in[\frac{1}{2},1]$,}
\end{cases}
$$
is a (well-defined) ${\mathcal S}^\mu$ map.
\end{lem}
\begin{proof}
Observe first that $\eta_m$ satisfies
\begin{itemize}
\item $\eta_m(0)=0$, $\eta_m(\tfrac{1}{2})=\frac{1}{2}$ and $\eta_m(1)=1$,
\item $\eta_m^{(\ell)}(\tfrac{1}{2})=0$ for each $\ell=1,\ldots,m-1$,
\item $\eta_m([0,\frac{1}{2}])=[0,\frac{1}{2}]$ and $\eta_m([\frac{1}{2},1])=[\frac{1}{2},1]$.
\end{itemize}
Thus, $\Psi^*(\Ss\times[0,1])\subset\Tt$, so $\Psi^*$ is a well-defined continuous semialgebraic function. 

Define $\Phi:=\Psi_1(\cdot,\frac{1}{2})=\Psi_2(\cdot,\frac{1}{2})$ and let $\Gamma_1:\Ss\times[0,\frac{1}{2}]\to\R^n$ and $\Gamma_2:\Ss\times[\frac{1}{2},1]\to\R^n$ be Nash functions such that $\Psi_i=\Phi+(\t-\frac{1}{2})\Gamma_i$. As
$$
\Psi:\Ss\times[0,1]\to\Tt,\ (x,t)\mapsto\begin{cases}
\Phi(x)+(t-\frac{1}{2})\Gamma_1(x,t)&\text{if $t\in[0,\frac{1}{2}]$,}\\
\Phi(x)+(t-\frac{1}{2})\Gamma_2(x,t)&\text{if $t\in[\frac{1}{2},1]$,}
\end{cases}
$$
we deduce
$$
\Psi^*:\Ss\times[0,1]\to\Tt,\ (x,t)\mapsto\begin{cases}
\Phi(x)+2^{m-1}(t-\frac{1}{2})^m\Gamma_1(x,\eta_m(t))&\text{if $t\in[0,\frac{1}{2}]$,}\\
\Phi(x)+2^{m-1}(t-\frac{1}{2})^m\Gamma_2(x,\eta_m(t))&\text{if $t\in[\frac{1}{2},1]$.}
\end{cases}
$$
As $\Gamma_1,\Gamma_2$ are Nash functions, they are locally bounded, so
$$
\Theta:\Ss\times[0,1]\to\R^n,\ (x,t)\mapsto\begin{cases}
2^{m-1}(t-\frac{1}{2})\Gamma_1(x,\eta_m(t))&\text{if $t\in[0,\frac{1}{2}]$,}\\
2^{m-1}(t-\frac{1}{2})\Gamma_2(x,\eta_m(t))&\text{if $t\in[\frac{1}{2},1]$}
\end{cases}
$$
is a continuous semialgebraic map on $\Ss\times[0,1]$ that is Nash on $\Ss\times([0,1]\setminus\{\frac{1}{2}\})$. As $m-1\geq\mu$, we deduce $(\t-\frac{1}{2})^{m-1}\Theta$ is an ${\mathcal S}^\mu$ map. Define 
$$
\varphi:\Ss\times[0,1]\to \R^n,\ (x,t)\mapsto\Phi(x).
$$
As $\varphi$ is Nash on $\Ss\times[0,1]$, we conclude that $\Psi^*=\varphi+(\t-\frac{1}{2})^{m-1}\Theta$ is an ${\mathcal S}^\mu$ map, as required.
\end{proof}

\begin{proof}[Proof of Corollary \em\ref{homocor2}]
By Theorem \ref{pushdiffeo} there exists a Nash map
$$
\sigma:\Qq\times [0,\tfrac{1}{2}]\to\Qq,\ (x,t)\mapsto\sigma_t(x):=\sigma(x,t)
$$
such that $\sigma_0=\id_{\Qq}$ and $\sigma_t(\Qq)\subset\Int(\Qq)$ for each $t\in(0,\tfrac{1}{2}]$. Let $(\Omega,\nu)$ be a Nash tubular neighborhood of $\Int(\Qq)$ in $\R^n$ and denote $h:=\sigma_{\frac{1}{2}}\circ f$. As $\R^n\setminus\Omega$ is closed and $h(\Ss)\subset\Int(\Qq)\subset\Omega$, the continuous semialgebraic function
$$
\veps:\Ss\to\R,\ x\mapsto\tfrac{1}{2}\dist(h(x),\R^n\setminus\Omega)
$$
is by \cite[Prop.2.2.8]{bcr} strictly positive. By Theorem \ref{thm1} there exists a Nash map $g:\Ss\to\Int(\Qq)$ that is $\veps$-close to $h$ with respect to the ${\mathcal S}^{\mu}$ topology. Consider the continuous semialgebraic homotopy
$$
\Phi:\Ss\times[\tfrac{1}{2},1]\to\R^n,\ (x,t)\mapsto(2-2t)h(x)+(2t-1)g(x).
$$
As $\|h(x)-g(x)\|<\veps(x)$ for each $x\in\Ss$, we have 
\begin{align}\label{homodist2}
\begin{split}
\|h(x)-\Phi_t(x)\|&=\|h(x)-(2-2t)h(x)-(2t-1)g(x)\|\\
&=|2t-1|\|h(x)-g(x)\|\leq\|h(x)-g(x)\|<\veps(x)=\tfrac{1}{2}\dist(h(x),\R^n\setminus\Omega),
\end{split}
\end{align}
so $\Phi_t(x)\in\Omega$ for each $(x,t)\in\Ss\times[\tfrac{1}{2},1]$. As $\nu:\Omega\to\Int(\Qq)$ is a Nash retraction, 
$$
\Psi:\Ss\times[\tfrac{1}{2},1]\to\Int(\Qq),\ (x,t)\mapsto\nu\circ \Phi(x,t)
$$
is a well-defined ${\mathcal S}^{\mu}$ homotopy between $h$ and $g$. Let $m>\mu$ be an odd integer and let $\eta_m:[0,1]\to [0,1],\ t\mapsto\frac{(2t-1)^m}{2}+\frac{1}{2}$, which is a Nash homeomorphism. By Lemma \ref{detallito}
$$
\Gamma:\Ss\times[0,1]\to\Qq,\ (x,t)\mapsto \begin{cases}
\sigma(f(x),\eta_m(t)), &\text{if } 0\leq t\leq\frac{1}{2},\\
\Psi(x,\eta_m(t)), &\text{if }\frac{1}{2}\leq t\leq 1
\end{cases}
$$
is the required ${\mathcal S}^{\mu}$ homotopy between $f$ and the Nash map $g:\Ss\to\Int(\Qq)$.
\end{proof}

\subsection{Further developments}
We next prove Lemmas \ref{fdlem1} and \ref{fdlem2} already stated in the Introduction.

\begin{proof}[Proof of Lemma \em\ref{fdlem1}]
Suppose $\Tt$ is not locally connected by analytic paths. Then there exists a point $p\in\Tt$ and an open semialgebraic neighborhood $W_0\subset\R^n$ of $p$ such that 
for each open semialgebraic neighborhood $W\subset W_0$ of $p$, the intersection $\Tt\cap W$ is not connected by analytic paths. Observe that $p$ is a non-isolated point of $\Tt$. By \cite[Thm.9.3.6]{bcr} there exists $\veps>0$ such that $\ol{\Bb}_n(p,\veps)\subset W_0$ and a semialgebraic homeomorphism $\varphi:\ol{\Bb}_n(p,\veps)\to\ol{\Bb}_n(p,\veps)$ such that:
\begin{itemize}
\item $\|\varphi(x)-p\|=\|x-p\|$ for every $x\in\ol{\Bb}_n(p,\veps)$.
\item $\varphi|_{\sph^{n-1}(p,\veps)}$ is the identity map.
\item $\varphi^{-1}(\Tt\cap\ol{\Bb}_n(p,\veps))$ is the cone with vertex $p$ and basis $\Tt\cap\sph^{n-1}(p,\veps)$.
\end{itemize}
Thus, $\Tt$ is locally connected by continuous semialgebraic maps at $p$. We claim: {\em $\Tt_\veps:=\Tt\cap\ol{\Bb}_n(p,\veps)$ is pure dimensional}. 

Suppose $\Tt_\veps$ is not pure dimensional. As $\varphi^{-1}(\Tt_\veps)$ is the cone with vertex $p$ and basis $\Tt\cap\sph^{n-1}(p,\veps)$, we deduce that $\Tt_\veps$ is not pure dimensional at $p$. Thus, there exist two closed semialgebraic subsets $\Rr_1$ and $\Rr_2$ of $\Tt_\veps$ such that $\dim(\Rr_1)<\dim(\Rr_2)$, $\Rr_1\cap\Rr_2=\{p\}$ and $\Rr_1\cup\Rr_2=\Tt_\veps$. Let $Y_1$ be the Zariski closure of $\Rr_1$, which has dimension $\dim(\Rr_1)$. In particular, $\Rr_1\setminus\Rr_2=\Rr_1\setminus\{p\}$ and $\Rr_2\setminus Y_1$ are semialgebraic sets adherent to $p$ of dimensions $\dim(\Rr_1)$ and $\dim(\Rr_2)$.

By the Nash curve selection lemma \cite[Prop.8.1.13]{bcr} there exist Nash arcs $\gamma_1:[-1,0]\to(\Rr_1\setminus\Rr_2)\cup\{p\}$ and $\gamma_2:[0,1]\to(\Rr_2\setminus Y_1)\cup\{p\}$ such that $\gamma_i(0)=p$, $\gamma_1([-1,0))\subset\Rr_1$ and $\gamma_2((0,1])\subset\Rr_2\setminus Y_1$. The semialgebraic map $\gamma:=\gamma_1*\gamma_2:[-1,1]\to\Rr_1\cup(\Rr_2\setminus Y_1)\cup\{p\}$ is continuous. Let $\psi:[-1,1]\to[-1,1]$ be a bijective ${\mathcal S}^\mu$ map such that $\psi(t)={\rm sign}(t)t^{2\mu}$ around $0$. The composition $\gamma\circ\psi:[-1,1]\to\Rr_1\cup(\Rr_2\setminus Y_1)\cup\{p\}\subset\Tt$ is an ${\mathcal S}^\mu$ map. Observe that $\gamma$ does not admit a close Nash approximation $\eta:[-1,1]\to\Tt$, because a part of the image of $\eta$ should be contained in $Y_1$, so by the identity principle for Nash functions the whole image should be contained in $Y_1$, which is a contradiction, because another part of the image of $\eta$ should be contained in $\Rr_2\setminus Y_1$. Thus, $\Tt_\veps$ is pure dimensional, as claimed.

Let $X\subset\R^n$ be the Zariski closure of $\Tt_\veps$ in $\R^n$. Let $\widetilde{X}\subset\C^n$ be the complexification of $X$ and let $\Sing(X):=\Sing(\widetilde{X})\cap\R^n$. Let $\Reg(\Tt_\veps)$ be the interior of $\Tt_\veps\setminus\Sing(X)$ in $X\setminus\Sing(X)$, which is a Nash manifold of the same dimension as $\Tt_\veps$. As $\Tt_\veps$ is pure dimensional, $\Reg(\Tt_\veps)$ is dense in $\Tt_\veps$. Recall also that by \cite[Main Thm.1.8 \& Lem.7.4]{fe2} any semialgebraic set between a connected Nash manifold and its closure is connected by analytic paths. Thus, by \cite[Prop.7.8]{fe2} there exist two connected components $M_1$ and $M_2$ of $\Reg(\Tt_\veps)$ such that for each pair of points $y_1\in M_1$ and $y_2\in M_2$ there exists no analytic path between $y_1$ and $y_2$.

Let $\alpha:[-1,1]\to\Tt_\veps$ be a continuous semialgebraic path such that $\alpha(-1)=y_1$ and $\alpha(1)=y_2$ (recall that $\Tt_\veps$ is connected by semialgebraic paths). By \cite[\S1.2.2]{fe3} there exists a finite set $\eta(\alpha)\subset[-1,1]$ such that $\alpha|_{[-1,1]\setminus\eta(\alpha)}$ is a Nash function. Let us `reparameterize' $\alpha$ in order to transform it into an element of ${\mathcal S}^\mu([-1,1],\Tt)$.

Pick a value $t_0\in\eta(\alpha)$. Assume $t_0\in(-1,1)$. The case of the extremes $\{-1,1\}$ of the interval $[-1,1]$ is similar, but simpler and we leave the details to the reader. After the translation $t-t_0$, we may assume that $t_0=0$. By \cite[p.166]{bcr} the ring $\displaystyle{\lim_{\longrightarrow}}_{\delta>0}{\mathcal S}^0((0,\delta],\R)$ is isomorphic to the field of algebraic Puiseux series $\R[[\t^*]]_{\rm alg}$. Let $p$ be the smallest positive integer such that $\alpha(\t^p)\in\R[[\t]]_{\rm alg}$. Let $q$ be the smallest positive integer such that $\alpha(-\t^q)\in\R[[\t]]_{\rm alg}$. We have to change $t$ around $t_0$ by 
$$
\theta_{t_0}:t\mapsto\begin{cases}
t_0-(t_0-t)^{q(\mu+1)}&\text{if $t<t_0$,}\\
t_0+(t-t_0)^{p(\mu+1)}&\text{if $t>t_0$.}
\end{cases}
$$
We proceed analogously with all the values of $\eta(\alpha)$ and we construct, using an ${\mathcal S}^\mu$ partition of unity associated to a suitable open semialgebraic partition of the interval $[-1,1]$, an ${\mathcal S}^\mu$ function $\theta:[-1,1]\to[-1,1]$ such that $\theta(-1)=-1$, $\theta(1)=1$ and $\theta$ is around each point $t_0\in\eta(\alpha)$ equal to $\theta_{t_0}$. The composition $\alpha\circ\theta:[-1,1]\to\Tt_\veps$ is an ${\mathcal S}^\mu$ path such that $(\alpha\circ\theta)(-1)=y_1$ and $(\alpha\circ\theta)(1)=y_2$.

As $\Tt$ is a $({\mathcal N},\mu)$-${\tt ats}$, there exist a Nash path $\beta:[-1,1]\to\Tt$ close to $\alpha$ in the ${\mathcal S}^\mu$-topology of ${\mathcal S}^\mu([-1,1],\Tt)$. Thus, $\beta$ is a Nash path inside $\Tt_\veps$ that connects a point of $M_1$ with a point of $M_2$ (because $M_1$ and $M_2$ are open semialgebraic subsets of $\Tt$), which is a contradiction. Consequently, $\Tt$ is locally connected by analytic paths, as required. 
\end{proof}

\begin{proof}[Proof of Lemma \em\ref{fdlem2}]
As $\Tt$ is locally connected by analytic paths, it is locally pure dimensional \cite[Lem.7.1]{fe2}, so the local dimension function on $\Tt$ is locally constant. As $\Tt$ is connected, it is pure dimensional. Let $X\subset\R^n$ be the Zariski closure of $\Tt$ in $\R^n$. Let $\widetilde{X}\subset\C^n$ be the complexification of $X$ and let $\Sing(X):=\Sing(\widetilde{X})\cap\R^n$. Let $\Reg(\Tt)$ be the interior of $\Tt\setminus\Sing(X)$ in $X\setminus\Sing(X)$, which is a Nash manifold of the same dimension as $\Tt$. Let $\Tt_1,\ldots,\Tt_r$ be the connected components of $\Reg(\Tt)$. As $\Tt$ is connected, we find inductively (after reordering the indices $i=2,\ldots,r$) points $p_i\in\Tt\cap\cl(\Tt_i)\cap\bigcup_{j=1}^{i-1}\cl(\Tt_j)$ for $i=2,\ldots,r-1$. As $\Tt$ is locally pure dimensional, we deduce that all the $\Tt_i$ have the same dimension. As $\Tt$ is locally connected by analytic paths at $p_i$, there exists an open semialgebraic neighborhood $W_i$ of $p_i$ such that $\Tt\cap W_i$ is connected by analytic paths. Let $x_i\in W_i\cap\Tt_i$ and $y_i\in W_i\cap\bigcup_{j=1}^{i-1}\Tt_j$. Let $\alpha_i:[-1,1]\to W_i\cap\Tt$ be an analytic path such that $\alpha_i(-1)=x_i$ and $\alpha_i(1)=y_i$. By \cite[Lem.2.9]{fe2} we may assume that $\alpha_i$ is a Nash path.

Using the previous Nash manifolds $\Tt_i$ and the Nash paths $\alpha_i:[-1,1]\to W_i\cap\Tt$ we find (maybe repeated) Nash manifolds $\Ss_1,\ldots,\Ss_m\in\{\Tt_1,\ldots,\Tt_r\}$ and Nash arcs $\beta_j:[-1,1]\to\Ss_{j-1}\cup\{q_j\}\cup\Ss_j$ where $q_j\in\cl(\Ss_{j-1})\cap\cl(\Ss_j)$ for $j=2,\ldots,m$ such that $\Ss_1=\Tt_1$, $\Ss_m=\Tt_r$ and $\{\Ss_1,\ldots,\Ss_m\}=\{\Tt_1,\ldots,\Tt_r\}$. Pick two points $x,y\in\Tt$. We may assume there exists indices $1\leq k\leq\ell\leq m$ such that $x\in\cl(\Ss_k)$ and $y\in\cl(\Ss_\ell)$. By \cite[Main Thm.1.8]{fe3} there exist a Nash path $\alpha:[0,1]\to\Tt$ such that $\alpha(0)=x$ and $\alpha(1)=y$. Consequently, $\Tt\subset\R^n$ is connected by analytic paths, as required.
\end{proof}

We next present an enlightening example of a semialgebraic set connected by analytic paths that is not locally connected by analytic paths.

\begin{example}\label{counter}
Let 
\begin{multline*}
\Tt:=\{(4\x^2-\y^2)(4\y^2-\x^2)\geq0,\y\geq0,\x^2+\y^2\leq 4\}\\
\cup\{(4\x^2-\y^2)\leq0,(\x^2+\y^2-1)(\x^2+\y^2-4)\leq 0,\y\geq0\}\subset\R^2, 
\end{multline*}
which is a (path) connected semialgebraic set, see Figure \ref{fig11}. As we have seen in the proof of Lemma \ref{fdlem1}, every semialgebraic set is locally connected by continuous semialgebraic paths. The reader can check that the ${\mathcal S}^\mu$ semialgebraic map $\alpha:[-1,1]\to\Tt,\ t\mapsto(t^{2\mu+1},|t|t^{2\mu})$ admits no Nash approximation. We claim: {\em The semialgebraic set $\Tt$ is connected by analytic paths, but not locally connected by analytic paths at the origin.} 

\begin{figure}[!ht]
\begin{center}
\begin{tikzpicture}[scale=0.75]

\draw[fill=gray!80, fill opacity=0.3,thick=1.5pt,draw] (4.5,1) -- (0.5,3) arc (153.43494882292201:26.56505117707799:4.47213595499958) -- (7,2.25) -- (4.5,1) -- (5.75,3.525) arc (63.43494882292202:116.43494882292201:2.8);

\draw[thick=1.5pt] (4.5,1) -- (0.5,3);
\draw[thick=1.5pt] (4.5,1) -- (8.5,3);
\draw[thick=1.5pt] (4.5,1) -- (3.25,3.525);
\draw[thick=1.5pt] (4.5,1) -- (5.75,3.525);

\draw[->] (4.5,0) -- (4.5,6);
\draw[->] (0,1) -- (9,1);

\draw[fill=black] (4.5,1) circle (0.75mm);
\draw[fill=black] (5.5,2) circle (0.75mm);
\draw[fill=black] (3.5,2) circle (0.75mm);

\draw (6.1,2.4) node{\footnotesize$(\veps,\veps)$};
\draw (2.9,2.4) node{\footnotesize$(-\veps,\veps)$};

\draw (5,5) node{\small$\Tt$};

\end{tikzpicture}
\end{center}
\caption{Semialgebraic set $\Tt\subset\R^2$\label{fig11}}
\end{figure}
\end{example}
\begin{proof}
The semialgebraic $\Tt$ is connected by analytic paths, because its interior is a connected Nash manifold (use \cite[Thm.9.3.6]{bcr}). Let us check that $\Tt$ is not locally connected by analytic paths at the origin. Pick the points $p_1:=(-\veps,\veps),p_2:=(\veps,\veps)\in\Tt\cap\Bb_2((0,0),2\veps)$ for some $0<\veps<\frac{1}{4}$ and assume that there exists an analytic path $\alpha:[0,1]\to\Tt\cap\Bb_2((0,0),2\veps)$ such that $\alpha(0)=p_1$ and $\alpha(1)=p_2$. Consider the closed semialgebraic sets $\Cc_1:=\Tt\cap\Bb_2((0,0),2\veps)\cap\{\x\leq0\}$ and $\Cc_2:=\Tt\cap\Bb_2((0,0),2\veps)\cap\{\x\geq0\}$, which satisfy $\Tt\cap\Bb_2((0,0),2\veps)=\Cc_1\cup\Cc_2$. Both $\Cc_1$ and $\Cc_2$ are convex, so they are connected by analytic paths (in fact, they are connected by segments) and $\Cc_1\cap\Cc_2=\{(0,0)\}$. Define $\Cc_i^*:=\{\lambda w:\ w\in\Cc_i,\ \lambda\in\R\}$ for $i=1,2$. Note that $\Tt\cap\Bb_2((0,0),2\veps)\cap\{\x<0\}=\Cc_1\setminus\{(0,0)\}$ and $\Tt\cap\Bb_2((0,0),2\veps)\cap\{\x>0\}=\Cc_2\setminus\{(0,0)\}$ are pairwise disjoint open subsets of $\Tt$. We have $0\in\alpha^{-1}(\Cc_1\setminus\{(0,0)\})$ and $1\in\alpha^{-1}(\Cc_2\setminus\{(0,0)\})$, so $t_0:=\inf(\alpha^{-1}(\Cc_2\setminus\{(0,0)\}))>0$. As $\alpha$ is a (non-constant) analytic path, $t_0\in\cl(\alpha^{-1}(\Cc_1\setminus\{(0,0)\}))\cap\cl(\alpha^{-1}(\Cc_2\setminus\{(0,0)\}))$ and $\alpha(t_0)=(0,0)$. As $\alpha^{-1}(\Cc_1\setminus\{(0,0)\})$ and $\alpha^{-1}(\Cc_2\setminus\{(0,0)\})$ are pairwise disjoint open subsets of $[0,1]$, there exists $\delta>0$ such that 
$$
\alpha((t_0-\delta,t_0))\subset\Cc_1\setminus\{(0,0)\}\quad\text{and}\quad\alpha((t_0,t_0+\delta))\subset\Cc_2\setminus\{(0,0)\}.
$$
The tangent direction to $\im(\alpha|_{(t_0-\delta,t_0+\delta)})$ at $\alpha(t_0)=(0,0)$ is the line generated by the vector 
$$
w=\lim_{t\to t_0}\frac{\alpha(t)-\alpha(t_0)}{(t-t_0)^k}=\begin{cases}
\lim_{t\to t_0^{+}}\frac{\alpha(t)-(0,0)}{(t-t_0)^k}\in\Cc_1^*\setminus\{(0,0)\},\\
\lim_{t\to t_0^{-}}\frac{\alpha(t)-(0,0)}{(t-t_0)^k}\in\Cc_2^*\setminus\{(0,0)\},
\end{cases}
$$
where $k$ is the multiplicity of $t_0$ as a root of $\|\alpha\|$. This is a contradiction (because $\Cc_1^*\cap\Cc_2^*=\{(0,0)\}$), so $\Tt\cap\Bb_2((0,0),2\veps)$ is not connected by analytic paths for each $0<\veps<\frac{1}{4}$. Thus, $\Tt$ is not locally connected by analytic paths at the origin.
\end{proof}

\bibliographystyle{amsalpha}

\end{document}